\documentclass[12pt]{amsart}
\usepackage{amsfonts}
\usepackage{amsmath, amsthm, amssymb,color}
\usepackage{latexsym}
\usepackage{amsaddr}
\usepackage{graphicx}
\usepackage{graphics}
\usepackage{epsfig}
\usepackage[english]{babel}
\usepackage{paralist}
\usepackage{multirow}
\usepackage[T1]{fontenc}
\numberwithin{equation}{section}
\allowdisplaybreaks

\newtheorem{lemma}{Lemma}[section]

\newtheorem{theorem}[lemma]{Theorem}
\newtheorem{proposition}[lemma]{Proposition}
\newtheorem{definition}[lemma]{Definition}
\newtheorem{corollary}[lemma]{Corollary}
\newtheorem{example}[lemma]{Example}
\newtheorem{exercise}[lemma]{Exercise}
\newtheorem{remark}[lemma]{Remark}
\newtheorem{fig}[lemma]{Figure}
\newtheorem{tab}[lemma]{Table}

\newcommand{\bth}{\begin{theorem}}
\newcommand{\ethe}{\end{theorem}}
\newcommand{\bre}{\begin{remark}\em }
\newcommand{\ere}{\end{remark}}
\newcommand{\ble}{\begin{lemma}}
\newcommand{\ele}{\end{lemma}}
\newcommand{\bde}{\begin{definition}}
\newcommand{\ede}{\end{definition}}
\newcommand{\bco}{\begin{corollary}}
\newcommand{\eco}{\end{corollary}}
\newcommand{\bpr}{\begin{proposition}}
\newcommand{\epr}{\end{proposition}}

\newcommand{\bexer}{\begin{exercise}}
\newcommand{\eexer}{\end{exercise}}

\newcommand{\bexam}{\begin{example}\rm  }
\newcommand{\eexam}{ \end{example}}

\newcommand{\bfi}{\begin{fig}}
\newcommand{\efi}{\end{fig}}

\newcommand{\btab}{\begin{tab}}
\newcommand{\etab}{\end{tab}}

\def\E{{\mathbb{E}}}
\def\P{{\mathbb{P}}}
\def\R{{\mathbb{R}}}

\def\N{{\mathbb{N}}}

\def\Y{{\mathbb{Y}}}
\def\Z{{\mathbb{Z}}}
\def\m{{\bf{m}}}

\def\I{{\mathcal{I}}}

\def\X{{\mathcal{X}}}
\def\CY{{\mathcal{Y}}}

\def\B_e{B_{\eta}(e)}

\renewcommand{\a}{\alpha }
\renewcommand{\b}{\beta}

\newcommand{\del}{\delta}
\newcommand{\g}{\gamma}
\newcommand{\eps}{\varepsilon}

\newcommand{\8}{\infty}

\newcommand{\ov}{\overline}

\newcommand{\wt}{\widetilde}

\newcommand{\blue}{\color{darkblue} }
\definecolor{darkblue}{rgb}{0,0,1}
\definecolor{darkgreen}{rgb}{0,1,0}
\definecolor{darkred}{rgb}{1, 0,0}

\newcommand{\garch}{{\rm GARCH}$(1,1)$}

\newcommand{\rv}{random variable}

\newcommand{\bfPi}{\mbox{\boldmath$\Pi$}}

\newcommand{\beao}{\begin{eqnarray*}}
\newcommand{\eeao}{\end{eqnarray*}\noindent}

\newcommand{\beam}{\begin{eqnarray}}
\newcommand{\eeam}{\end{eqnarray}\noindent}

\newcommand{\beqq}{\begin{equation}}
\newcommand{\eeqq}{\end{equation}\noindent}

\newcommand{\bce}{\begin{center}}
\newcommand{\ece}{\end{center}}

\newcommand{\barr}{\begin{array}}
\newcommand{\earr}{\end{array}}

\newcommand{\eqd}{\stackrel{d}{=}}

\newcommand{\vague}{\stackrel{\lower0.2ex\hbox{$\scriptscriptstyle
                    \it{v} $}}{\rightarrow}}
\newcommand{\weak}{\stackrel{\lower0.2ex\hbox{$\scriptscriptstyle
                    \it{w} $}}{\rightarrow}}
\newcommand{\what}{\stackrel{\lower0.2ex\hbox{$\scriptscriptstyle
                    \it{\hat{w}} $}}{\rightarrow}}

\newcommand{\bdis}{\begin{displaymath}}
\newcommand{\edis}{\end{displaymath}\noindent}

\newcommand{\wh}{\widehat}
\newcommand{\vep}{\varepsilon}

\newcommand{\con}{convergence}

\newcommand{\bfx}{{\bf x}}
\newcommand{\bfX}{{\bf X}}
\newcommand{\bfB}{{\bf B}}

\newcommand{\bfy}{{\bf y}}
\newcommand{\bfA}{{\bf A}}

\newcommand{\bfW}{{\bf W}}

\newcommand{\bfI}{{\bf I}}

\newcommand{\bali}{\begin{align}}
\newcommand{\eali}{\end{align}}

\newcommand{\Ind}{\mathbf{1}_}
\newtheorem{thm}[equation]{Theorem}

\usepackage{ascmac}	
\begin{document}
\bibliographystyle{alpha}

\title[Tails of bivariate stochastic recurrence equation]{Tails of bivariate stochastic recurrence equation with triangular matrices }
\today
\author[E. Damek]{Ewa Damek}
\address{Institute of Mathematics
University of Wroc\l aw\\ Pl.  Grunwaldzki 2/4, 50-384, Wroc\l aw, Poland}  
\email{edamek@math.uni.wroc.pl}
\author[M. Matsui]{Muneya Matsui}
 \address{Department of Business Administration, Nanzan University \\
18 Yamazato-cho Showa-ku Nagoya, 466-8673, Japan}
\email{mmuneya@nanzan-u.ac.jp}

\begin{abstract} 
We study bivariate stochastic recurrence equations with triangular matrix coefficients and we
characterize the tail behavior of their stationary solutions $\bfW =(W_1,W_2)$. 
Recently it has been observed that $W_1,W_2$ 
may exhibit regularly varying tails with different indices, which is in contrast to well-known Kesten-type results. 
However, only partial results have been derived. 
Under typical ``Kesten-Goldie'' and ``Grey'' conditions, 
we completely characterize tail behavior of $W_1,W_2$. The tail asymptotics we obtain has not been observed in previous settings of stochastic recurrence equations. \\
\vspace{2mm} \\
{\it Key words. }\ Stochastic recurrence equation, regular variation, 
Kesten's theorem,  autoregressive models, triangular matrix. 
\end{abstract}
\subjclass[2010]{Primary 60G70, 60G10, 60H25, Secondary 62M10, 91B84}
\maketitle 
\section{Introduction}
We consider the stochastic recurrence equation (SRE)
\begin{equation}\label{affine} 
\bfW_t = \bfA_t\bfW_{t-1}+\bfB_t, \quad t\in \N,
\end{equation}
where $(\bfA_t,\bfB_t)$ is an i.i.d.\ sequence, $\bfA_t$ are $d\times d$
matrices, $\bfB_t$ are vectors and $\bfW_0$ is an initial distribution
independent of the sequence $(\bfA_t, \bfB_t)$. Iteration of
\eqref{affine} generates a Markov chain $(\bfW_t)_{t\geq0}$ that is not necessarily stationary.
Under mild contractivity hypotheses (see e.g. \cite{bougerol:picard:1992a,brandt:1986})
the sequence $\bfW_t$ converges in law to a random
vector $\bfW$ that is the unique solution of the equation
\begin{equation*}
\bfW \stackrel{d}{=} \bfA \bfW+\bfB, 
\end{equation*} where $\bfW$ is independent
of $(\bfA,\bfB)$ 
and the equation is meant in law. Here $(\bfA,\bfB)$ is a generic element of the sequence $(\bfA_t,\bfB_t)$.
If we put $\bfW_0=\bfW$ then the chain $\bfW_t$ becomes
stationary. Moreover, extending the set of indices to $\Z$
and taking an i.i.d.\ sequence $(\bfA_t, \bfB_t)_{t\in \Z}$ we can have a strictly stationary causal solution $\bfW_t$ to the equation
\begin{equation*}
\bfW_t = \bfA_t\bfW_{t-1}+\bfB_t, \quad t\in \Z . 
\end{equation*}
It is given by
\begin{equation*}
\bfW _t=\sum _{i=-\8}^t\bfA_t\cdots \bfA_{i}\bfB_{i-1}+\bfB_i\stackrel{d}{=}\bfW.
\end{equation*} 

The stochastic iteration \eqref{affine} and its variants have been studied since the seventies,  they have found numerous applications in finance, insurance, telecommunication, time series analysis and they still attract a lot of attention. In particular, the tail behavior of the stationary solution $\bfW$ is of vital interest to the 
risk management (\cite{embrechts:kluppelberg:mikosch:1997}, \cite[Sec. 7.3]{mcneil:fre:embrechts:2015}).
It provides also moment conditions for statistical models which are crucial in parameter estimation problems (e.g. parameter estimation for GARCH processes).
 For an overview
we refer the reader to 
Buraczewski et al. \cite{buraczewski:damek:mikosch:2016}). 



The first set of conditions implying regular behavior of $\bfW$ in the sense of \eqref{directions} below 
was formulated by Kesten \cite{kesten:1973}. 
Since then, the Kesten condition and its extensions have been used 
to characterize tails in various situations, 
an essential feature being 
the same tail behavior in all directions \cite{alsmeyer:mentmeier2012, guivarch:lepage:2015, buraczewski:damek:guivarch2009}. 
To put it simply, there is an $\alpha>0$ and a measure on $\R ^d$ being the weak limit of
\begin{equation}\label{directions}
x^{\a}\P (x^{-1}\bfW \in \cdot ), \quad \mbox{when} \quad x\to \8.
\end{equation}
The behavior \eqref{directions} follows from certain irreducibility or homogeneity of the action 
of the group generated by the support of the law of $\bfA$: random shocks circulate over all directions, so that coordinate-wise tail behavior is the same.
However, this property is not necessarily shared by all models interesting both from theoretical and applied perspective 
\cite{horvath:boril:2016,matsui:mikosch:2016,matsui:pedersen:2019,mentemeier:wintenberger:2020,pedersen:wintenberger,sun:chen:2014}. 
Therefore, SREs with more general $\bfA$ are both challenging and desirable. 

Notice that already
for SREs with diagonal matrices $\bfA=diag(A_{11},\ldots,A_{dd})$, the tail indices of particular coordinates may be different: each coordinate satisfies the 
corresponding 
univariate SRE and they do not interact. This leads to ``non standard'' or ``vector valued'' regular variation  \cite{damek:2021,mentemeier:wintenberger:2020, resnick:2007}.
Then, naturally triangular matrices $\bfA$ occur, 
which provides SREs with partial interactions between coordinates,
and therefore considering them is a natural next step. However, the existing methods cannot be applied and a new approach is needed. It has been partly developed in 
\cite{damek:matsui:swiatkowski:2019, damek:zienkiewicz:2017, matsui:swiatkowski} and  now we propose a complete solution for the case of 
$2\times 2$ upper triangular matrices
$\bfA=[A_{ij}]$ 
(i.e.\ $A_{21}$ is the only one being identically zero). Even then the proof is quite involved and uses a broad range of methods that will be gradually explained.

Write $W=(W_1, W_2)$ and under natural conditions, 
we obtain that $W_1, W_2$ are regularly varying with possibly different indices. 
For $(\bfA ,\bfB)$ we assume either ``Kesten-Goldie'' 
condition, 
\begin{equation}\label{kes}
\E |A_{ii}|^{\a _i}=1\quad \E |B_i|^{\a _i}<\8 \quad \mbox{for some}\quad \a_ i>0
\end{equation}
or ``Grey'' condition,  
\begin{equation}\label{gr}
\E |A_{ii}|^{\a _i}<1\quad \mbox{and}\quad   B_i\ \mbox{regularly varying with index}\quad \a_ i>0.
\end{equation}
The regular variation of $W_2$ follows directly from the univariate results. 
The tail of $W_1$, however, 
is determined by all 
 the entries of $(\bfA , \bfB )$. We prove that 
\begin{equation}\label{asymone}
\P (\pm W_1>x)\sim c_{\pm}x^{-\min (\a _1,\a _2) }\ell (x), \quad \mbox{as}\ x\to \8 
\end{equation}
for an appropriate slowly varying function $\ell$ and $c_{+}+c_->0$ (see Theorems \ref{theorem:main1}, \ref{main1}, \ref{main}).  
Clearly, some extra integrability assumptions are needed similar to the ones usually used in the one dimensional case. 
They are formulated in Section \ref{mainresults}. Although for the ``Grey case'' $\ell$ comes from the regular behavior of $B_1, B_2$, 
for the ``Kesten-Goldie case'' the presence of slowly varying functions is due to mutual interaction between the entries of $(\bfA , \bfB )$. This   
is a novelty that has not yet been observed in the case of SRE. 
The latter 
phenomena appear only when $\a _1=\a _2=:\a $ and then $\ell (x)=(\log x)^{\beta }$,
with $\beta =1, \a $ or $\a \slash 2$. Most of our effort is concentrated on this case and already for $2 \times 2$
matrices the proof is very technical.
 
Under the setting \eqref{kes}, with all the entries of $\bfA, \bfB$ being positive and 
$\a _1\neq \a _2$, \eqref{asymone} was obtained in \cite{damek:matsui:swiatkowski:2019} with $\ell (x)=1$. 
Later on \eqref{asymone} was generalized to $d\times d$ matrices under the assumption $\E A_{ii}^{\a _i}=1$, $A_{ii}>0$,
 with $\a _1,...,\a _d$ being all different (\cite{matsui:swiatkowski}). Then
\begin{equation*}
\P (W_i>x)\sim c_ix^{-\wt \a _i},\quad c_i>0,
\end{equation*}
where $\wt \a _i $ depends on $\a _i,...,\a _d$. Moreover, the case
of $2\times 2$ matrices with $A_{11}=A_{22}>0$ was treated in \cite{damek:zienkiewicz:2017}. 
However, not much has been done when $\a _i=\a _j$ for some $i\neq j$ but $A_{ii}, A_{jj}$ are not equal almost surely.
There are only rough estimates 
\begin{equation}
c_i x^{-\wt \a _i}\leq \P (W_i>x)\leq C_i x^{-\wt \a _i}(\log x)^{\beta _i}
\end{equation}
(see \cite{WS}) again under assumption of positivity of $\bfA, \bfB$. It is not clear how to make use of both our approach and \cite{matsui:swiatkowski,WS} to obtain a definitive answer.    

There are various financial models that satisfy \eqref{affine} and 
in order to prove that the finite dimensional distributions of the corresponding stochastic process $X$ are regularly varying, the results \cite{kesten:1973,goldie:1991,buraczewski:damek:guivarch2009,alsmeyer:mentmeier2012,guivarch:lepage:2015} 
have been used intensively, e.g. in \cite{basrak:davis:mikosch:2002, matsui:mikosch:2016, matsui:pedersen:2019, pedersen:wintenberger}. 
The concept of regular variation is convenient to study extremal behavior of the process $X$ in terms of the maxima or extremal indices. For 
GARCH(p,q), bivariate GARCH(1,1) or BEKK-ARCH processes, when assumptions of \cite{kesten:1973, buraczewski:damek:guivarch2009,alsmeyer:mentmeier2012}
 or \cite{guivarch:lepage:2015} are applicable,  regular variation was studied in 
\cite{basrak:davis:mikosch:2002, matsui:mikosch:2016, matsui:pedersen:2019} and \cite{pedersen:wintenberger},  
and conclusions for the extremal properties of $X$ have been obtained, see  \cite{starica:1999, matsui:mikosch:2016}. 

The BEKK-ARCH process, introduced by Engle and Kroner \cite{engle:kroner:1995} and originally defined by a non-affine recursion, has been written as \eqref{affine} by Pedersen and Wintenberger \cite{pedersen:wintenberger}. They studied the regular behavior when assumptions of \cite{alsmeyer:mentmeier2012} or \cite{buraczewski:damek:guivarch2009} are applicable. Recent results on SREs with diagonal matrices \cite{damek:2021, mentemeier:wintenberger:2020} allow to study diagonal BEKK models typically used in 
finance due to their relatively simple parametrization (see Bauwens et al. \cite{bauwens:laurent:rombouts:2006}). Also BEKK-ARCH with triangular matrices has been of interest (see \cite{matsui:pedersen:2019}) and then the results of this paper as well as the multivariate ones \cite{matsui:swiatkowski} are applicable. 

The remainder of the paper is organized as follows. 
In Section \ref{section:model:stationary}, we describe the model and prove existence of a unique stationary solution. 
The main results 
are presented in Section \ref{mainresults}, where we make distinction between 
$\alpha_1\neq \alpha_2$ and $\alpha_1= \alpha_2$, the latter being much more involved. The proofs are contained in Sections \ref{sec:pr:theorem:main1} and \ref{proofs} respectively.
In Section \ref{sec:uni:SRE} we provide a formula for the Goldie constants in the univariate SRE, which is frequently 
used in the previous sections and which is interesting in itself. 
For readers convenience a list of frequently used symbols is given just before  acknowledgments. 

We close this section by introducing some notation used throughout the paper.
For functions $f,g:\R\to \R$, $f(x) \sim g(x) $ means that $\lim_{x\to \infty} f(x)/g(x)\to 1$. 
For a real number $a$ we write $a^+=\max (a,0)$, $ a^-=-\min (a,0)$ and moreover, $\log^+ a=\log(1\vee a)$. 
For a vector $\bfx\in \R^d$, $|\bfx|$ denotes its Euclidean norm and
for a $d\times d$ matrix $\bfA$ we use the matrix norm;
\[
 ||\bfA||=\sup_{\bfx\in\R^d,\,|\bfx|=1}|\bfA \bfx|. 
\]


\section{Bivariate stochastic recurrence equations}
\label{section:model:stationary}
We start with description of the model as well as the 
conditions for stationarity of the related time series.

\subsection{The model}
We consider the bivariate SRE;
\begin{align}\label{bivSRE}
 \bfW_t =\bfA_t \bfW_{t-1}+\bfB_t, \quad t \in \Z, 
\end{align}
where 
\beao
\bfW_t= 
\begin{pmatrix}
W_{1,t}  \\ W_{2,t}   
\end{pmatrix},\quad
\bfA_t=
\begin{pmatrix}
 A_{11,t} & A_{12,t} \\
0 & A_{22,t}
\end{pmatrix}
\quad \mathrm{and}\quad \bfB_t= \begin{pmatrix}
 B_{1,t}  \\ B_{2,t} , 
\end{pmatrix}
\eeao
and $(\bfA_t, \bfB_t)$ an i.i.d. sequence. Unlike in \cite{damek:matsui:swiatkowski:2019} we do not assume here any restriction on the sign of the entries of matrices and vectors, they are just real numbers.
It is convenient to write the SRE also in a coordinate-wise form; 
\begin{align}
 W_{1,t} & =
 A_{11,t}W_{1,t-1} + D_t, \label{uniSRE1} \\
 W_{2,t} & = A_{22,t} W_{2,t-1} + B_{2,t}, \label{uniSRE2} 
\end{align}
where 
\begin{equation}\label{D}
D_t:=B_{1,t}+A_{12,t}W_{2,t-1}. 
\end{equation}
For further convenience we denote
for $t\in\mathbb{Z}$, 
\begin{align*}
  \bfPi_{t,s} & =\bfA_t\cdots \bfA_s,\,t\ge s,\quad
 \bfPi_{t,s}=\bfI,\,t<s\quad \mathrm{and}\quad \bfPi_t =\bfPi_{t,1}, \\
  \Pi_{t,s}^{(i)}& =\Pi _{j=s}^t A_{ii,j},\,t\ge s,\,i=1,2\quad
 \mathrm{and}\quad \Pi_{t,s}^{(i)}=1,\,t<s\quad \mathrm{and}\quad \Pi_t^{(i)}=\Pi_{t,1}^{(i)},
\end{align*}
where $\bfI$ is the bivariate identity matrix. 

\subsection{Stationarity}\label{stac}
Starting from \cite{kesten:1973} there is a series of results \cite{brandt:1986}, \cite{bougerol:picard:1992a} for the existence of stationary solution to SRE (see also \cite[Sec.2.1]{buraczewski:damek:mikosch:2016} for a review). The notion of the ``so called''  top Lyapunov exponent 
\beao
\gamma=\inf_{n\ge 1} n^{-1} \E \log \| 
\bfPi_{n}\| 
\eeao
associated with the sequence $(\bfA _t)$ is essential. If $\gamma$ is strictly negative and 
\begin{equation}\label{logmoments}
\E (\log ^+|\bfB|+\log ^+\| \bfA\| )<\8,  
\end{equation}
then SRE \eqref{bivSRE} has a unique strictly stationary solution
 (\cite{bougerol:picard:1992a}, see also  
 \cite[Theorem 4.1.4]{buraczewski:damek:mikosch:2016}) given by the infinite
 series, 
 \begin{align}
\label{solbibSRE}
  \bfW_t=\sum_{i=-\infty}^t \bfPi_{t,i+1}\bfB_i. 
 \end{align}
When matrices $\bfA$ are block triangular, we refer to \cite{GGO} and \cite{S} for
conditions that imply negativity of $\gamma$. For the bivariate
case we will use the following statement.   
\begin{proposition}
\label{condi:stationarity}
 Assume $\E \log |A_{11}|<0,\,\E \log |A_{22}|<0,\,\E \log^+ |\bfB|<\infty
 $ and $\E\|\bfA\|^\varepsilon <\infty$ for some $\varepsilon>0$. Then $\gamma $ is strictly negative.
\end{proposition} 
Proposition \ref{condi:stationarity} was proved in 
\cite[Proposition 2.1]{damek:matsui:swiatkowski:2019} for matrices and vectors with positive entries, 
but the proof in the general case is the same. 
Although the assumptions are a little bit stronger than necessary 
but they are anyway satisfied in our main theorems (Theorems \ref{theorem:main1},\ref{main} 
and \ref{main1}). 
For further discussion we refer to \cite{damek:matsui:swiatkowski:2019}.
 

{\it Due to stationarity we may restrict our attention to the tails of $W_{i,0}$. 
The subscript $0$ in $A_{ij,0}$, $B_{i,0}$ and
$W_{i,0}$, etc. will be sometimes omitted and we will write
$A_{ij}$, $B_{i}$ and $W_{i}$ for generic random variables. 
}  

\subsection{Componentwise decomposition}\label{component}
We will work on the 
component-wise representation of the solution 
$\bfW_t =(W_{1,t},W_{2,t})$ given by
\begin{align}
\label{componentwiseSol1}
W_{1,t} &= \sum_{i=1}^\infty \Pi_{t,t+2-i}^{(1)} D_{t+1-i},\\
\label{componentwiseSol2}
W_{2,t} &= \sum_{i=1}^\infty \Pi_{t,t+2-i}^{(2)} B_{2,t+1-i}, 
\end{align}
which are well-defined. 
 The expressions 
\eqref{componentwiseSol1} and \eqref{componentwiseSol2} of the solution  
may be proved in various ways but under our hypotheses the proof is particularly simple. Indeed, the assumptions we require in the 
main theorems (\ref{theorem:main1},\ref{main} and \ref{main1}) imply existence of  $0<\eps <1 $ such that
\begin{equation}\label{eps}
\E |A_{ii}|^{\eps}<1,
\end{equation}
\begin{equation}\label{epss}
\E |\bfB|^{\eps}<\8 \quad \mbox{and}\quad \E |A_{12}|^{\eps}<\8.
\end{equation}
Then, first we see by the triangle inequality that $\E |W_2|^{\eps }<\8$, hence $\E |D|^{\eps }<\8$. 
Again by the triangle inequality
$\E |W_1|^{\eps }<\8$ follows. 
Thus, \eqref{componentwiseSol1} and \eqref{componentwiseSol2} are convergent. 
Substituting \eqref{componentwiseSol1} and \eqref{componentwiseSol2} to \eqref{bivSRE} we see that 
$(W_{1,t},W_{2,t})$ satisfies the equation and by uniqueness 
the stationary solution $\bfW_t$ satisfies \eqref{componentwiseSol1} and \eqref{componentwiseSol2}. 
For more details we refer to Section 2.2 of \cite{damek:matsui:swiatkowski:2019}. 

In order to study tail asymptotics, we will further decompose $W_{1,t}$.  
Let $\wh W_1$ and $\wt W_1$ be respectively unique
stationary solutions of SREs 
\begin{align}
\label{componentwiseSRE1}
\wh W_{1,t} &=  A_{11,t} \wh W_{1,t-1}+B_{1,t}, \\
\label{componentwiseSRE2}
\wt W_{1,t} &=  A_{11,t} \wt W_{1,t-1}+ \wt D_{t}, \quad \wt D_{t}=A_{12,t}W_{2,t-1}. 
\end{align}
Then
\begin{align}
\label{decomp1}
W_{1,t} = \wh W_{1,t} + \wt W_{1,t},\quad t \in\Z. 
\end{align} 
By the same reasoning as above,
both \eqref{componentwiseSRE1} and \eqref{componentwiseSRE2} have unique
solutions respectively and they may be written as 
\begin{align}
\label{srestasol1}
\wh W_{1,t} =  \sum_{i=1}^\infty \Pi_{t,t+2-i}^{(1)}B_{1,t+1-i}, 
\end{align}
\begin{equation}
\label{srestasol2}
\wt W_{1,t} =  \sum_{i=1}^\infty \Pi_{t,t+2-i}^{(1)} A_{12,t+1-i} W_{2,t-i}, 
\end{equation}
where the series converge absolutely almost surely. 
The advantage of this approach is that we may compare the tails of both $\wh W_{1,t}$ and $\wt W_{1,t}$ and decide which of them is 
the heavier and determines the asymptotics of $W_{1,t}$. 

\section{Main results}
\label{mainresults}
As already mentioned, our assumptions are modeled on those typically used for the univariate equation $X=AX+B$.
They are either Kesten-Goldie assumption when the tail of 
the stationary solution is determined by $A$, or Grey assumption when $B$ plays the dominant role. 
\begin{itembox}[l]{$\mathcal A(\a )$: Kesten-Goldie assumption}
	$\bullet$ There exists $\alpha>0$ such that
	$\E|A|^\alpha=1,\,\E|B|^\alpha<\infty$ and $\E |A|^\alpha \log^+
	|A|<\infty$. \\
	$\bullet$\ $\P(Ax+B= x)<1$ for every $x \in \R$. \\
	$\bullet$\ The conditional law of $\log |A|$ given
	$\{A\neq 0\}$ is non-arithmetic.
\end{itembox}
\begin{itembox}[l]{$\mathcal B(\a )$: Grey assumption} 
	There exist $\alpha, \eta >0$ such that
	$\E|A|^\a<1,\,\E|A|^{\alpha +\eta} <\infty$ and 
	\begin{align}
	\label{slowvary:def1}
	\P(B>x)\sim p_\a\, x^{-\a }\ell (x)\quad \text{and}\quad \P(-B>x)\sim q_\a\, x^{-\a}\ell (x)
	\end{align}
	with $p_\a ,\,q_\a \ge 0,\,p_\a +q_\a=1$, where $\ell (x)$ is a
	slowly varying function. 
\end{itembox}

We assume that $(A_{ii}, B_{i})_{i=1,2}$ satisfy $\mathcal{A}(\a _i)$ or $\mathcal{B}(\a _i)$.
Then the tail behavior of $W_2$ follows directly form the univariate results (Theorem \ref{lemma:4447BDM}): 
\begin{align}
\label{tail:w_2}
 \P(\pm W_2>x) \sim c_{2,\pm}\, x^{-\a _2}\ell_2(x), \quad \mbox{as}\ x\to\infty
\end{align}
where $\ell_2=1$ if $\mathcal{A} (\a _2)$ holds and $\ell_2$ is a slowly varying function if $\mathcal{B} (\a _2)$ holds. 
Here constants $c_{2,+}$ and $c_{2,-}$ 
are given by
\eqref{constants:uni} with $(A,B,X)=(A_{22},B_2,W_2)$. Strict positivity of $c_{2,+}+c_{2,-}$ means that
$$
\lim _{x\to \8}\P (|W_2|>x)x^{\a _2}\ell _2(x)=c_{2,+}+c_{2,-}>0$$ 
and if $\mathcal{B}(\a _2)$ holds, it is straightforward from \eqref{constants:uni}. For $\mathcal{A}(\a _2)$ an extra argument is needed, see \cite{goldie:1991}. 

The same estimate for $W_1=\wh W_1+\wt W_1$ is more delicate. 
Since $\wh W_1$ satisfies \eqref{componentwiseSRE1}, it is regularly varying with index $\alpha_1$. 
Then the tail of $W_1$ is determined by $\min (\alpha_1 ,\alpha_2)$. 
When $\a _1 <\a _2$, then the noise component $\wt D$ of $\wt W_1$ in \eqref{componentwiseSRE2} satisfies $\E|\wt D|^{\alpha_1}$, so the same holds for 
$D$ in \eqref{componentwiseSol1}. 
Thus the tail of $W_1$ is determined by the autoregressive part $A_{11}$. 
When $\a _2< \a _1$ the situation is quite the opposite 
and $W_2$ via the noise term $\wt D$ of $\wt W_{1}$ determines the tail of $W_1$. 
The same happens when $\a _1=\a _2$, though then analysis of $\wt W_{1}$ in far more complicated as 
explained after Theorem \ref{main1} and at the beginning of Section \ref{subsec:pf:main}.

In what follows, we start with the case $\a_1\neq \a_2$ (Theorem \ref{theorem:main1}) and then study the case $\a _1=\a _2=\a$ 
(Theorems \ref{main1} and \ref{main}). 
Since $\wh W_1$ has the same univariate form \eqref{componentwiseSRE1}, it is immediate to see 
\begin{align}
\label{tail:w_1}
\P(\pm \wh W_1>x) \sim c_{1,\pm}\, x^{-\a _1}\ell_1(x)
\end{align}
where $\ell_1=1$ if $\mathcal{A} (\a _1)$ holds and $\ell_2$ is a slowly varying function if $\mathcal{B} (\a _1)$ holds. 
The constants are again determined by 
%
\eqref{constants:uni} with $(A,B,X)=(A_{11},B_1,W_1)$.


To formulate the results precisely we need some notation. Let
\begin{equation}\label{rho}
	\rho_{1}=\E|A_{11}|^{\alpha_1} \log |A_{11}|, \quad  \rho _2=\E|A_{22}|^{\alpha _2}\log |A_{22}|, 
\end{equation}
\begin{equation}\label{Mn}
M_n=\sum_{i=1}^n \Pi_{0,2-i}^{(1)}A_{12,1-i}
\Pi_{-i,1-n}^{(2)}, 
\end{equation}
and
\begin{equation*}
 w_{n,\pm}=\E (M_n^{\pm})^{\a _2} \quad \text{and}\quad w_n=\E |M_n|^{\a _2}.
\end{equation*} 
We will prove later on that, if $\a _1> \a _2$ 
the limits 
\begin{equation}\label{w}
w=\lim _{n\to \8} w_n 
>0\quad \mbox{and}\quad w _{\pm}=\lim _{n\to \8} w_{n,\pm}
\end{equation}
exist and $w, w_{\pm}$  appear in the tail constants of Theorem \ref{theorem:main1} (Table \ref{table:constants}). 

\medskip

Now we are ready to formulate the main results. We have $4$ patterns of the tail behavior of $W_1$ 
depending on whether $\mathcal A$ or $\mathcal B$ are satisfied and which of $\alpha_1,  \alpha_2$ is larger. 
\begin{theorem}
\label{theorem:main1}
Suppose that $(A_{ii}, B_{i})_{i=1,2}$ satisfy $\mathcal{A}(\a _i)$ or $\mathcal{B}(\a _i)$ for $\a _1\neq \a _2$,
and moreover,  
\begin{equation}\label{C2}
  \P (A_{12}= 0)<1 \quad \mbox{and}\quad \E |A_{12}|^{
  \alpha_1 \wedge \alpha_2 }<\infty.
\end{equation}
 Then, if $\a _1<\a_2$,
\begin{align}\label{firstv}
 \P(\pm W_1>x) \sim \Big \{
\begin{array}{ll}
\bar c_\pm\, x^{-\alpha_1}  & \mbox{if} \quad\mathcal{A}(\a _1)\\
 \bar c_{\pm }\, x^{-\alpha_1}\ell_1(x)    & \mbox{if} \quad\mathcal{B}(\a _1) 
\end{array}
\end{align}
and if $\a _1>\a_2$,
\begin{align}\label{secondv}
\P(\pm W_1>x) \sim \Big \{
\begin{array}{ll}
\tilde c_\pm\, x^{-\alpha_2}  & \mbox{if} \quad\mathcal{A}(\a _2) \ \& \ \P (A_{22}=0)=0 
\\
\tilde c_{\pm }\,x^{-\alpha_2}\ell_2(x) & \mbox{if} \quad\mathcal{B}(\a _2)  
\end{array},
\end{align}
where constants are given in Table \ref{table:constants} and 
$\ell_i,\,i=1,2$ are slowly varying functions defined from \eqref{slowvary:def1}. \\
Moreover, 
 \begin{equation}\label{possum}
  \tilde c_++\tilde c_->0 \ \text{in}\, \eqref{secondv} \quad \mbox{and} \quad 
\bar c_{+}+\bar c_{-}>0 \ \mbox{in}\ \mathcal{B}(\a _1)\ \text{of}\ \eqref{firstv}.  
   \end{equation}
Finally, if all the entries in $\bfA , \bfB$ are non negative or 
\begin{align}
\label{posiconst:barc}
 \E \big[|\wh W_1|^{\a_1} -|A_{11}\wh W_1|^{\a_1} \big] \neq \E \big[|\wt W_1|^{\a_1}-|A_{11}
 \wt W_1|^{\a_1} \big]
\end{align} 
holds, then 
\begin{equation}\label{positivitybarc}
\bar c_+ +\bar c_->0 \ \mbox{in} \ \mathcal{A}(\a _1)\  \text{of}\ \eqref{firstv}.  
\end{equation}
\end{theorem}
The proof of Theorem \ref{theorem:main1} is given in
Section \ref{sec:pr:theorem:main1}.
\begin{remark}
$(\rm{i})$ Under $\mathcal{A}(\a _1)$ or $\mathcal{A}(\a _2)$ with $A_{ij}, B_i\geq 0,\,i,j=1,2$ Theorem \ref{theorem:main1} was proved in \cite{damek:matsui:swiatkowski:2019}.\\
$(\rm{ii})$ 
If $\mathcal{A}(\a _1)$ holds and the entries of $\bfA , \bfB$ are not all non negative, it is hard to prove that 
$$
\lim _{x\to \8}\P (|W_1|>x)x^{\a _1}=\bar c_++\bar c_->0.$$
The point is that in $W_{1,t}=A_{11,t}W_{1,t-1}+D_t$, $W_{1,t-1}$ and $D_t$ are dependent so that the argument of Goldie \cite{goldie:1991} does not work. We may use the decomposition $W_1=\wh W_1+\wt W _1$. If \eqref{posiconst:barc} holds then 
$\lim _{x\to \8}\P (|\wh W_1|>x)x^{\a}
\neq \lim _{x\to \8}\P (|\wt W_1|>x)x^{\a }$ and so \eqref{positivitybarc} follows. Otherwise the question remains open. 
Notice that	if we replace $A_{12,t}$ by $a A_{12,t}$, or $B_2$ by $aB_2$, $\,a\in \R$ in the original SRE
	\eqref{bivSRE}, then the solution of \eqref{componentwiseSRE2} becomes $\wt W_{1,t}=a
	\sum_{i=1}^\infty \Pi_{0,2-i}^{(1)} A_{12,1-i} W_{2,-i}$ and we obtain 
	\[
	|a|^{\a_1} \E\big [|\wt W_{1}|^{\a_1}-|A_{11} \wt W_{1}|^{\a_1} \big]
	\]
on the left hand side of \eqref{posiconst:barc}. Hence \eqref{posiconst:barc} may be  violated only at two values of $a$ which gives an impression that equality in 
	\eqref{posiconst:barc} is rather exceptional. 
Both sides of \eqref{posiconst:barc}, or equivalently, the tail constants 
are not easy to calculate $($see \cite[Section 3]{buraczewski:damek:mikosch:2016}$)$. Although there are several attempts 
$($cf. \cite{mikosch:samorodnitsky:tafakori}$)$ the results are far from being exhaustive.
\end{remark}

\begin{table}[h]
 \caption{Constants in Theorem \ref{theorem:main1}}
\begin{tabular}{c|c|c}
\multicolumn{2}{c|}{ Conditions }  & Constants \\
\hline
\multirow{4}{*}{$\mathcal A(\a_1)$} & \multirow{2}{*}{$A_{11}\ge
     0\ a.s.$} &
\multirow{2}{*}{$ \bar c_\pm= (\alpha_1 \rho_{1})^{-1} \E [((D_0+A_{11} W_1)^{\pm})^{\alpha_1}-((A_{11}W_1)^\pm)^{\alpha_1}] $} \\
& & \\
&\multirow{2}{*}{$\P(A_{11}<0)>0$}  & 
\multirow{2}{*}{$ \bar
 c_\pm= (2 \alpha_1 \rho_{1})^{-1} \E [|D_0+A_{11} W_1|^{\alpha_1}-|A_{11}W_1|^{\alpha_1}]$ }\\
& & \\
\hline 
\multirow{4}{*}{$\mathcal A(\a_2)$}& \multirow{2}{*}{$A_{22}\ge 0\ a.s.$}& \multirow{2}{*}{$ \tilde c_\pm= c_{2,+} w_{\pm}+c_{2,-}w_{\mp}$}\\
& & \\
& \multirow{2}{*}{$\P(A_{22}<0)>0$}& \multirow{2}{*}{$ \tilde
 c_+= \tilde c_-= c_2 w/2,\quad c_2/2 =c_{2,+}=c_{2,-} $}\\
& & \\ \hline
\multicolumn{2}{c|}{\multirow{2}{*}{$\mathcal B(\alpha_1)$}} & 
\multirow{2}{*}{$ \bar c_\pm= \frac{1}{2} \Big\{
\frac{1}{1-\E |A_{11}|^{\alpha_1} } \pm \frac{p_{\alpha_1}-q_{\alpha_1}}{1-\E(A_{11}^+)^{\alpha_1}+\E(A_{11}^-)^{\alpha_1}}\Big\} $} \\ 
\multicolumn{2}{c|}{} &  \\\hline
\multicolumn{2}{c|}{\multirow{2}{*}{$\mathcal B(\alpha_2)$}} & \multirow{2}{*}{$\tilde c_\pm = \sum_{i=1}^\infty
 w_{i,\pm}p_{\alpha_2}+w_{i,\mp} q_{\alpha_2}$} \\
\multicolumn{2}{c|}{}  & \\ \hline
\end{tabular}
\label{table:constants}
\end{table}

The case $\mathcal A(\a_i)$ with $\alpha_1 = \alpha_2 =\alpha$,
i.e. the case $\E|A_{ii}|^{\alpha}=1,\,i=1,2$ is much more involved.  
We distinguish two cases depending on whether $A_{11}=A_{22}\ a.s.$ or not. 
Accordingly we need the following common and specific conditions. 
\begin{itembox}[l]{Common assumptions in Theorems \ref{main1} and \ref{main}}
\begin{itemize}
	\item[\rm{[A1]}]\quad $\E \log |A_{ii}|<0$, $\E |A_{ii}|^{\a}=1$, $i=1,2$.
	\item[\rm{[A2]}]\quad there is $\eta >0$ such that $\E |A_{ij}|^{\a
		+\eta}+\E |B_i|^{\a +\eta}<\8,\,i,j=1,2$. 
        \item[\rm{[A3]}]\quad  $A_{22}\neq 0\,a.s.$
        \item[\rm{[A4]}]\quad $\log |A_{11}|$ is non lattice.
\end{itemize}
\end{itembox}
\begin{itembox}[l]{Assumptions specific to Theorem \ref{main} (Case $\P(A_{11}\neq A_{22})>0$)}
\begin{itemize}
\item[\rm{[A5]}]\quad $\log |A_{22}|$, $\log (|A_{11}||A_{22}|^{-1})$ are non arithmetic.
\item[{\rm [A6]}]\quad there is $\eta >0$ such that $\E |A_{11}|^{\a +\eta} |A_{22}|^{-\eta}<\8$, $\E |A_{12}|^{\a +\eta} |A_{22}|^{-\eta}<\8$. 
\end{itemize}
\end{itembox}

Before going to main results, we provide some intuition. 
When $\a _1=\a _2$, the tail of $$\wt W_1=\sum_{i=1}^\infty \Pi_{t,t+2-i}^{(1)} A_{12,t+1-i} W_{2,t-i}$$ is heavier than that of $\wh W_1$ and
the partial sum 
\begin{equation*}
\sum_{i=1}^{n} \Pi_{t,t+2-i}^{(1)} A_{12,t+1-i} W_{2,t-i}\quad \mbox{with}\quad n= \lfloor c \log x \rfloor
\end{equation*}
provides the asymptotics (see the beginning of Section \ref{proofs}). Our basic observation is
\begin{align}
\label{tailw1:breiman}
 \P(\wt W_1>x) \sim C_+ \E(M_n^+)^\a \P(W_2>x) + C_- \E(M_n^-)^\a \P(-W_2>x)\ \mbox{with}\ n= \lfloor c \log x \rfloor, 
\end{align}
where $c,C_+,C_-$ are positive constants depending on $\a$ and $A_{ii}$. 
Then as in Theorem \ref{theorem:main1}, the behavior of $\E (M_n^{\pm})^{\a}$ in \eqref{Mn} again plays the crucial role.
We consider two cases $\P(A_{11}\neq A_{22})>0$ and 
$\P(A_{11}\neq A_{22})=0$. 
\vspace{2mm}

Observe that in the case $[A_{11}=A_{22}\, a.s.]$, in view of [A3],
$A_{11}\neq 0$ a.s and to state the result we set 
\begin{align}
\label{notations1:muai1}
 \mu= \E A_{11}^{-1} A_{12} |A_{11}|^\a \quad \text{and}\quad \sigma^2 =\E (A_{12}A_{11}^{-1})^2 |A_{11}|^\a, 
\end{align}
and 
\begin{align}
\label{notations2:muai1}
 \mathcal{C}= \sigma^\a \rho_1^{-\a/2} \E|N|^\a \quad \text{with} \ N \ \text{the standard normal r.v.} 
\end{align}
\begin{theorem}
\label{main1} Assume ${\rm [A1 \mathchar`-A4]}$. 
	Suppose further that $A_{11}=A_{22}\,a.s.$ and $\sigma^2<\infty$. \\ 
If $\mu=0$ then 
\begin{align*}
                \lim _{x\to \infty}\P (\pm  W_1>x)x^{\a}(\log x)^{-\a \slash 2} &= c_2\,\mathcal{C}/2, \\
 \text{so that}\quad  \lim _{x\to \infty}\P (|W_1|>x)x^{\a}(\log x)^{-\a \slash 2} &=c_2\,\mathcal{C}. 
\end{align*}
If $\mu \neq 0$ then 
\begin{align*}
\lim _{x\to \infty}\P (\pm W_1>x)x^{\a}(\log x)^{-\a } &=
\left \{
\begin{array}{cl}
c_{2,\pm}\mu^{\a }\rho_1^{-\a}  & \mbox{if}\quad \mu>0 \\
c_{2,\mp}|\mu|^{\a }\rho_1^{-\a} & \mbox{if}\quad \mu <0, 
\end{array}
\right. \\
\text{so that} \quad 
\lim _{x\to \infty}\P (|W_1|>x)x^{\a}(\log x)^{-\a } &=c_2\rho_1^{-\a}|\mu |.
\end{align*}
\end{theorem}
The proof is given in Section \ref{pf:main1}. Theorem \ref{main1} with $A_{11}>0$ was proved in \cite{damek:zienkiewicz:2017},  
but presently the proof has been considerably simplified and we do not assume positivity of $A_{11}$.

The extra function $(\log x)^\beta,\,\beta = \alpha,\alpha/2$ that appears in the tails of $W_1$,  comes from $\E(M_{n}^\pm)^\alpha$ if $A_{11}>0$ a.s. or from $\E|M_{n}|^\alpha$ if not
(see \eqref{tailw1:breiman}). The latter is kind of surprising because it is not visible from \eqref{tailw1:breiman} and we will come to it later. First let us explain the simpler case of $A_{11}>0$.  
 
To study $\E(M_{n}^\pm)^\alpha$, we change the measure using $A_{11}>0$, 
i.e. $\E_\a[\,\cdot\,]=\E[A_{11}^\a\,\cdot\,]$. Then
$$\E(M_{n}^\pm)^\alpha = \E _{\a}\left ((U_1+\cdots +U_n)^{\pm }\right )^{\a },$$
where $U_1,\ldots, U_n$ are iid  with generic r.v. 
$U=A_{12}A_{11}^{-1}$ (Lemmas \ref{lem4} and \ref{lem5}).
In view of \eqref{notations1:muai1}, $\E_\a U=\mu$ and $\E_\a U^2=\sigma^2$. 
Then, if $\mu=0$, the central limit theorem with convergence of moments is applied and we have 

\begin{equation*}
 \lim_{n\to\infty} (\rho_1 n)^{-\a/2} \E(M_n^\pm)^\a =\mathcal C/2,
\end{equation*}
If $\mu\neq 0$ we replace $U$ by $U-\mu $ 
and obtain 
\begin{equation*}
 \lim_{n\to \infty} (\rho_1 n)^{-\a} \E(M_n^\pm)^\a =(\mu^\pm)^\a
\end{equation*} 
 If $\P (A_{11}<0)>0$ then luckily only $\E |M_n|^{\a}$ is needed, $\E _{\a}[\cdot ] = \E [|A_{11}|^{\a}] $ and the scheme is the same. Clearly we are not able to touch $\E _{\a}(M_n^{\pm})^{\a}$ when $A_{11}$ is signed but it is not an obstacle as explained in the proof of Theorem \ref{main1term}. Condition $\rm{[A3]}$ is related to the change of measure.

We proceed to the second case. 
\begin{theorem}\label{main} Suppose $\P (A_{11}\neq A_{22})>0$. Under Assumptions ${\rm [A1 \mathchar`-A6]}$, 
	\begin{equation}
	\label{plusminus0}
	\lim _{x \to \infty}\P (\pm W_1>x)x^{\a}(\log
	x)^{-1}=\mathcal D \a \rho _1^{-1},\end{equation}
	where
	\begin{equation*}
	\mathcal D = \left \{
	\begin{array}{cl}
	c_2c_R/2 & \mbox{if}\quad \P (A_{22}<0)>0 \ \mbox{or}\ [\,\P (A_{22} >0)=1\ \&\ \P (A_{11} <0)>0\,]\\
	c_{2,\pm} c_{R,+}+c_{2,\mp}c_{R,-} & \mbox{if}\quad \P (A_{22} >0)= \P (A_{11} \geq 0)=1
	\end{array}
	\right.
	\end{equation*} 
	and
	\begin{equation}\label{cR}
	c _{R,\pm}= \lim _{n\to \8}(\a n)^{-1} \E (M_n^{\pm})^{\a }\quad \mbox{and}\quad  c_R=\lim _{n\to \8} (\a n)^{-1}\E |M_n|^{\a }>0 .
	\end{equation}	
	In particular, 
	\begin{equation}\label{both}
	\lim _{x \to \infty}\P (|W_1|>x)x^{\a}(\log x)^{-1}=\a c_2c_R \rho _1 ^{-1}>0. \end{equation}
\end{theorem}
The proof is given in Section \ref{subsec:pf:main}. 

As before, the extra function $\log x$ in \eqref{plusminus0} 
 comes from $\E(M_n^{\pm})^\a$ if $A_{22}>0$ or $\E|M_n|^\a$ if not. 
 In order to show $\E(M_n^{\pm})^\a , \E|M_n|^\a \sim c \log x$, 
we change the measure similarly as before, i.e. $\E_\a[\,\cdot\,]=\E[|A_{22}|^\a\,\cdot\,]$ or $\E[A_{22}^\a\,\cdot\,]$ if $A_{22}>0$. 
But this time we consider SRE $X_t=V_t X_{t-1} + U_t$ under the measure $\P_\a$ (Lemma \ref{lem2}), 
where generic r.v.'s for an iid sequence $(V_t,U_t)$ are given by  
$V=A_{11}A_{22}^{-1}$ and $U=A_{12}A_{22}^{-1}$ respectively. We are able to transform
 $\E(M_n^\pm)^\a$ into $\E _{\a}(\X_n^{\pm})^{\a}$ where $\X_n$ converges to 
the stationary solution $\X$ to $X_t=V_t X_{t-1} + U_t$ (see \eqref{limits} in Theorem \ref{constant}). 

\begin{remark}
It is known that by using the standard regular variation, 
we can not capture the joint regular behavior of $\bfW$ properly, 
if tail orders are different as in our case. 
Therefore, several suggestions have been made such as ``non-standard regular variation'' by Resnick \cite{resnick:2007} or 
``vector scaling regular variation'' recently by Mentemeier and Wintenberger \cite{mentemeier:wintenberger:2020}. In particular, the latter 
notion was applied to \eqref{affine} with diagonal matrices 
in \cite[Section 6]{mentemeier:wintenberger:2020} and  
 \cite{damek:2021}. 
Then, the next natural question is how to characterize the tail of $\bfW$ in the triangular case in terms of  ``non-standard'' or ``vector scaling regular variation''. 
It is left for the future.      
\end{remark}

\section{Proof of Theorem \ref{theorem:main1}}
\label{sec:pr:theorem:main1}
\noindent
We give the proof separately for the $4$ cases depending on assumptions $\alpha_1 \lessgtr \alpha_2$ and $\mathcal A$ or $\mathcal B$.
Unless specified, $C$ denotes a positive constant whose value is not of interest. 
 \\
{\bf [\,Case $\alpha_1<\alpha_2$, $\mathcal{A}(\a _1)$\,] } 
Observe that 
the stationary solution $W_{1,0}$ of SRE \eqref{uniSRE1} 
satisfies 
\[
W_{1,0} = D_0+ A_{11,0} W_{1,-1}, 
\]
where $W_{1,-1}$ has the same law as $W_{1,0}$ and independent of
$A_{11,0}$. In view of Theorem 2.3 and Lemma 9.4 of 
Goldie (1991)\nocite{goldie:1991}, we may conclude that if
$\P(A_{11}\ge 0)=1$, the conditions 
\begin{align}
\label{eq:goldie+}
\int_0^\infty \big|
\P(\pm W_{1,-1}>x)-\P(\pm A_{11,0}W_{1,-1}>x)
\big| x^{\alpha_1-1} dx <\infty
\end{align}
respectively imply 
\begin{align*}
& \lim_{x\to\infty} \P(\pm W_{1,0}>x)x^{\alpha_1} 
= \bar c_\pm,  
\end{align*}
where
\begin{align*}
c_\pm=\rho_{1}^{-1}
\int_0^\infty \big(
\P(\pm W_{1,-1}>x)-\P(\pm A_{11,0}W_{1,-1}>x)
\big) x^{\alpha_1-1} dx <\infty.
\end{align*}
Similarly if $\P(A_{11,0}<0)>0$ and both of  
\eqref{eq:goldie+} 
hold, then
\begin{align*}
& \lim_{x\to\infty} \P( \pm W_{1,0}>x)x^{\alpha_1} 
& = (2 \rho_{1})^{-1} \int_0^\infty
\big(\P(|W_{1,-1}|>x)-\P(|A_{11,0}W_{1,-1}|>x) \big)x^{\alpha_1-1}dx, 
\end{align*}
so that $\bar c_+=\bar c_-$. 
Therefore, what needs to be proved is \eqref{eq:goldie+}. 
Choose $\bar \alpha$ such that $\a _1<\bar \alpha<\alpha_2$, and then 
 $\E|W_2|^{\bar \alpha}<\infty$ so that $\E|D_0|^{\alpha_1}<\infty$. 
Then the proof is quite similar to that of Theorem 3.2 in \cite{matsui:pedersen:2019}
or Theorem 3.2 in \cite{damek:matsui:swiatkowski:2019}. 
We omit further details. 

Next we show that $\bar c_{+}+ \bar c_{-}$ in $\mathcal{A}(\a _1)$ is strictly positive. 
Notice that we cannot take the
Goldie's approach \cite[Theorem 4.1]{goldie:1991}. Since $(D_t)$ of \eqref{D}
is
a dependent sequence, the key of L\'evy inequality there does not
work.
We exploit the decomposition \eqref{decomp1}. It follows from Lemma
\ref{lemma:4447BDM} that the 
solution $\wh W_1$ of SRE \eqref{componentwiseSRE1} satisfies 
$$
\P( |\wh W_1|>x) \sim  c_1
x^{-\alpha_1}\quad  \mbox{with}\quad   c_1>0.$$
 Similarly, the solution $\wt W_1$ for SRE \eqref{componentwiseSRE2} satisfies 
 $$ 
\lim _{x\to \8}\P( |\wt W_1|>x)x^{\alpha_1}=  c' \geq 0.$$
Notice that
\begin{align*}
c_1&=( \alpha_1 \rho_{1})^{-1}\E  [|\wh W_1|^{\a _1}- |A_{11}\wh W_1|^{\a _1}], \\
c'&=( \alpha_1 \rho_{1})^{-1}\E  [|\wt W_1|^{\a _1}- |A_{11}\wt W_1|^{\a _1}].
\end{align*}
If $c_1>c'$, for a large positive $\zeta $ we write
\begin{equation*}
\P(|W_1|>x) \ge \P(|\wh W_1| >(1+\zeta)x)-\P(|\wt W_1| >\zeta x).
\end{equation*}
Hence
\begin{equation*}
\lim _{x\to \8}x^{\a _1}\P(|W_1|>x) \ge (1+\zeta)^{-\a _1}c_1- \zeta ^{-\a _1}c'>0
\end{equation*}
if $\zeta $ is sufficiently large. If $c'>c_1$, we proceed similarly. 
\hfill $\qed$ \\
\bigskip 

{\bf [\,Case $\alpha_1<\alpha_2$, $\mathcal{B}(\a _1)$\,]} 
There is $\bar \a : \a _1<\bar \a <\a _2$ such that
$\E |A_{11}|^{\bar \a}<1$. Indeed, $f(\beta)=\E |A_{11}|^{\beta}<\8 $ is well defined for $0\leq \beta \leq \a _1+\eta $ and continuous. So $f(\beta)\geq 1$ for all $\beta >\a _1$ is not possible. 
Hence by \eqref{srestasol2},
$\E |\wt W_1|^{\bar \a}<\8$ by Minkowski inequality or subadditivity and the conclusion follows from \eqref{tail:w_1}. \hfill $\qed$ 

\bigskip 
\noindent
{\bf [\,Case $\a _1>\a_2$, $\mathcal{A}(\a _2)$\,]} 
First we describe the constants $\omega _{\pm}$, $\omega$ that appear in \eqref{w}. We have the following lemma. 
\begin{lemma}\label{lemma:sreconstant}
	$\rm (a)$ Suppose that $\a _1> \a _2$, $\mathcal{A}(\a _2)$, \eqref{C2} and $\P (A_{22}=0)=0$ 
are satisfied. 
Then the limit 
\begin{align}
\label{def:limitabsw}
w=\lim_{n\to \infty} \E |M_n|^{\a_2} >0  
\end{align}
exists. \\
$\rm (b)$ If additionally $\P (A_{22}>0)=1$ then the limits
\begin{align}
\label{eq:limpm}
w_\pm = \lim_{n\to\infty} \E (M_n^
{\pm})^{\a_2} 
\end{align} 
exist.
\end{lemma} 
\begin{proof}
Notice that in view of \eqref{Mn}, we have 
	\begin{align*}
|M_n| 
=& \big |\Pi ^{(2)}_{0,1-n}\big |
\Big | 
\sum _{i=1}^n \Pi ^{(1)}_{0,2-i}(\Pi ^{(2)}_{0,2-i})^{-1}A_{12,1-i}A_{22,1-i}^{-1}\Big |\\
=&  \big |\Pi^{(2)}_{0,1-n}\big | \Big |
\sum _{i=1}^{n}V _0\cdots V _{2-i} U_{1-i} \Big |,	
\end{align*}
where
\begin{align}
\label{notation:changemeas}
V_i = A_{11,i}A_{22,i}^{-1},\quad U_i= A_{12,i}A_{22,i}^{-1}\quad
 \mbox{and for}\ i=1, \ V_0\cdots V_{2-i}=1.
\end{align}
Now we change the measure ; let $\mathcal F_n$ be the filtration
defined by the sequence $(\bfA_i,\bfB_i)$ : $\mathcal
F_n=\sigma((\bfA_i,\bfB_i)_{-n \le i\le 0})$. Then the expectation $ \E _{\a _2}$
w.r.t. the new probability measure $\P_{\a _2}$ is defined by 
\begin{equation}\label{changemeasure}
\E _{\a _2}[Z] = \E  \big[|\Pi ^{(2)}_{0,-n} |^{\a _2} Z\big], 
\end{equation}
where $Z$ is measurable w.r.t. $\mathcal F_n$. 
Then 
\begin{equation*}
\E _{\a _2}\Big (\sum _{i=1}^{\infty}|V _0\cdots V _{2-i}U_{1-i}|\Big)^{\a _2}<\infty .
\end{equation*}
Indeed, 
\[
 \E _{\a _2}|V_i| ^{\alpha_2} =\E |A_{11,i}|^{\alpha_2}<1, \quad  \E _{\a _2}|U_i| ^{\alpha_2} =\E |A_{12,i}|^{\alpha_2}<\infty  
\]
so that by subadditivity and Minkowski inequality, we have 
\begin{align*}
& \E _{\a _2}\big(
\sum_{i=1}^\8 |V_0\cdots V_{2-i} U_{1-i}|
\big)^{\alpha_2} \\
& \le \left \{
\begin{array}{ll}
\sum_{i=1}^\8 (\E _{\a _2}| V|^{\alpha_2})^{i-1} \E _{\a _2}|U|^{\alpha_2} &
\text{for}\quad \alpha_2\leq 1 \vspace{2mm}\\
\big(\sum_{i=1}^\infty 
( \E _{\a _2}| V|^{\alpha_2})^{(i-1)/\alpha_2}\big)^{\alpha_2}
 \E _{\a _2} |U|^{\alpha_2} & \text{for}\quad \alpha_2> 1 
\end{array}
\right. \\
& < \infty. 
\end{align*}
This proves, in particular, that the series 
$$
\sum_{i=1}^\8 |V_0\cdots V_{2-i} U_{1-i}|$$
converges a.s. and so does
$$
X_0=\sum_{i=1}^\8 V_0\cdots V_{2-i} U_{1-i}.$$ 
Therefore,
by the Lebesgue dominated convergence theorem
$$
\lim _{n\to \8}\E |M_n|^{\a_2}=\lim _{n\to \8} \E _{\a _2} \big |
\sum_{i=1}^n V_0\cdots V_{2-i} U_{1-i}\big |^{\alpha_2}= \E _{\a _2} \big |
\sum_{i=1}^{\8} V_0\cdots V_{2-i} U_{1-i}\big |^{\alpha_2}=:w.$$
For $w_\pm$ in the case 
$\P (A_{22}>0)=1$, we write
$$
M_n^{\pm}=\Pi ^{(2)}_{0,1-n}
\Big (
\sum _{i=1}^n \Pi ^{(1)}_{0,2-i}(\Pi ^{(2)}_{0,2-i})^{-1}A_{12,1-i}A_{22,1-i}^{-1}\Big )^{\pm}
$$
and proceed as before.

Finally, we check that 
$$
w= \E _{\a _2}|X_0|^{\a _2}\neq 0, $$
i.e. that $X_0\neq 0$ a.s. Notice that 
$$
X_t=\sum _{i=1}^{\8}V_t\cdots V_{t+2-i}U_{t+1-i}$$ 
is a stationary solution to the SRE: 
\begin{equation}\label{newsre}
X_t=V_tX_{t-1}+U_t\quad \mbox{under}\quad \P_{\a _2}
\end{equation}
and for every $t$, $X_t$ has the same law as $X_0$. (Here, as before, $V_t\cdots V_{t+2-i}=1$ for i=1.)
Suppose that \eqref{newsre} has a unique solution. Then $X_0=0$ implies $U=0$ a.s., which gives a contradiction. 
Indeed, uniqueness of the solution is guaranteed by two conditions:
$$
-\8 \le  \E _{\a _2}\log |V|<0\quad \mbox{and}\quad \E _{\a _2}\log ^+ |U|<\8,$$
see, 
e.g. Theorem 2.1.3 in \cite{buraczewski:damek:mikosch:2016}.
We have 
$$
 \E _{\a _2}\log ^+|U| \leq \a _2^{-1} \E _{\a _2}|U|^{\a _2}= \a _2^{-1} \E |A_{12}|^{\a _2}<\8 $$
and similarly $\E _{\a _2}\log ^+|V| \leq \a _2^{-1} \E |A_{11}|^{\a _2}<\8$.
To prove that $ \E _{\a _2}\log |V|<0$, let us define
\begin{align}
\label{convexityQ}
f(\beta )= \E _{\alpha_2}|V|^{\beta }, \quad \beta \in [0,\a _2].
\end{align}
Then $f', f''$ exist in $(0,\a _2)$ such that $f''\ge 0$. Since $f(0)=1$ and $f(\a _2)= \E |A_{11}|^{\a _2}<1$,  
either $\E _{\a _2}\log |V|$ is equal $-\8 $ or if it is finite then $f'(0)=\E _{\a _2}\log |V|<0$. 
\end{proof}

Let us return to the proof of {\bf [\,Case $\a _1>\a_2$, $\mathcal{A}(\a _2)$\,]}.  
Observe that $\wh W_1$ is regularly varying with
 index $\alpha_1$, i.e. 
\[
\lim _{x\to \8} \P(|\wh W_{1,0}| >x) x^{\alpha_1} \quad \mbox{or}\quad 
\lim _{x\to \8} \P(|\wh W_{1,0}| >x) x^{\alpha_1} \ell_1(x)^{-1}\quad \mbox{exists}.
\]
We are going to show that the tail of $\wt W_1$ is dominant, i.e under $\mathcal{A}(\a _2)$
\[
\lim _{x\to \8}\P(\pm \wt W_{1,0} >x)x^{\alpha_2} =\tilde c_{\pm}.
\]
 We
decompose $\wt W_1$ into three parts,
\begin{equation}
 \wt W_{1,0} = \Big( \underbrace{\sum_{i=1}^s}_{\wt Z_s} +\underbrace{\sum_{i=s+1}^\infty}_{\wt Z^s} \Big)\, \Pi_{0,2-i}^{(1)}
 A_{12,1-i}W_{2,-i} =: \underbrace{\wt Z_{s,1} + \wt Z_{s,2}}_{\wt Z_s} + \wt
 Z^s, \label{decomp:wtX3}
\end{equation}
where in the decomposition of $\wt Z_s$, we apply the iteration \eqref{uniSRE2} of $W_2$ until
 time $-s<-i$,
\begin{align}
\label{def:decomp:w2}
 W_{2,-i} = \Pi_{-i,1-s}^{(2)} W_{2,-s}+ {\sum_{k=0}^{s-i-1}
 \Pi^{(2)}_{-i,1-i-k} B_{2,-i-k}},
\end{align}
and substitute this into $\wt Z_{s}$, so that
\begin{align}
\label{decomp:wtXunders}
 \wt Z_s
 &= \underbrace{\sum_{i=1}^s \Pi_{0,2-i}^{(1)}A_{12,1-i}
 \Pi_{-i,1-s}^{(2)}
 W_{2,-s}}_{\wt Z_{s,1}} +
 \underbrace{\sum_{i=1}^s \Pi_{0,2-i}^{(1)} A_{12,1-i}
 \sum_{k=0}^{s-i-1}\Pi_{-i,1-i-k}^{(2)}B_{2,-i-k}}_{\wt Z_{s,2}}.
\end{align}

By comparing their tail behavior we specify the dominant term and the negligible ones.
The idea is then to study the tail behavior of each term in \eqref{decomp:wtX3}. First we show the general scheme of the proof. The detailed tail asymptotics of the dominant and negligible terms will be given later. Specifically, we are going to  show that there are constants $C>0,\,0<q<1$ such that for every $s$
\begin{equation}
\label{geometric}
\P (|\wt Z^s|>x)\leq C q^sx^{-\a_2}.\end{equation}
Moreover, for a fixed (but arbitrary) $s$,
\begin{equation}\label{second}
\lim _{x\to \8}\P (|\wt Z_{s,2}|>x)x^{\a _2}=0
\end{equation}
\noindent and
\begin{align}\label{first}
&\lim _{x\to \8}\P (\wt Z_{s,1}>x)x^{\a _2}=c_{2,+}w_{s,+}+c_{2,-}w_{s,-}:=c_{s,+},\\
&\lim _{x\to \8}\P (\wt Z_{s,1}<-x)x^{\a _2}=c_{2,-}w_{s,+}+c_{2,+}w_{s,-}:=c_{s,-}, 
\end{align}
where $c_{2,\pm}$ are those in \eqref{tail:w_2} and $w_{s,\pm}$ are those in \eqref{w}.
Moreover, we will prove that
\begin{align}
\label{bounded}
 \lim _{s\to \8}c_{s,+}=&\tilde c_+\quad \mbox{and} \quad \lim _{s\to \8}c_{s,-}=\tilde c_- \quad \mbox{exist}. 
\end{align}
Hence $\wt Z_{s,1}$ is the dominating term in \eqref{decomp:wtX3}. Now,
 using \eqref{decomp1} and \eqref{decomp:wtX3}, we have that
\begin{align}
\begin{split}
 \P(W_1>x) &\le \P(\wt Z_{s,1}>(1-3\varepsilon)x)+
\P(\wh W_1>\varepsilon x) +
\P(\wt
 Z_{s,2}>\varepsilon x) +\P(\wt Z^s >\varepsilon x), \\
 \P(W_1>x) &\ge \P(\wt Z_{s,1}>(1+3\varepsilon)x)-
\P(\wh W_1<-\varepsilon x) -
\P(\wt
 Z_{s,2}<-\varepsilon x) -\P(\wt Z^s <-\varepsilon x). 
\end{split}
\end{align}
Then after multiplying by $x^{\alpha_2}$ both sides, we
 take the limit when $x\to\infty$ and obtain 
\begin{align}
\label{ineq:main}
 (1+3\varepsilon)^{-\alpha_2} c_{s,+} - C \varepsilon^{-\alpha_2} q^s &\le \liminf_{x\to\infty}
 x^{\alpha_2} \P(W_1>x) \\
 &\le \limsup_{x\to\infty} x^{\alpha_2}\P(W_1>x) \nonumber \\
 &\le    (1-3\varepsilon)^{-\alpha_2} c_{s,+} + C\varepsilon^{-\alpha_2}  q^s. \nonumber
\end{align}
Finally, letting first $s\to \8$ and then $\eps \to 0$, we obtain 
\[
 \lim_{x\to \infty} x^{\alpha_2} \P(W_1>x)=\wt c_{+}. 
\]
By changing the sign in \eqref{decomp:wtX3} 
and inequalities \eqref{ineq:main}, namely considering $-W_1$, similarly we get 
\[
 \lim_{x\to\infty} x^{\alpha_2} \P(W_1<-x)=\tilde c_-. 
\]
Here $\wt c_\pm$ are those in \eqref{bounded} of which 
formulae will be given at the end of the proof.
It remains to prove \eqref{geometric}--\eqref{bounded}. The proof of
 \eqref{geometric} and \eqref{second} is the same as that in \cite{damek:matsui:swiatkowski:2019} and it is omitted.
 For \eqref{first} we observe that
\begin{align}
\label{def:decomp:Zn}
 \wt Z_{s,1} = M_s W_{2,-s},
\end{align}
where $M_s:= \sum_{i=1}^s \Pi_{0,2-i}^{(1)}A_{12,1-i} \Pi_{-i,1-s}^{(2)}$
 and $W_{2,-s}$ are independent. Then 
\begin{align*}
\P (\wt Z_{s,1}>x) 
=&\P (M_s W_{2,-s}>x, M_s >0, W_{2,-s}>0)+ \P (M_s W_{2,-s}>x, M_s <0, W_{2,-s}<0)
\end{align*} 
and by Breiman's lemma
\begin{align*}
\lim _{x\to \8}\P \big( M_s W_{2,-s}>x, M_s \gtrless 0, W_{2,-s} \gtrless 0 \big) x^{\a _2}=w_{s,\pm}c_{2,\pm} 
\end{align*} 
which implies \eqref{first}. Now we want to take the limit when $s\to \8$. If $\P (A_{22}>0)=1$ then
by Lemma \ref{lemma:sreconstant}
$$
\lim _{s\to \8}c_{s,+}=w_+c_{2,+}+w_-c_{2,-}=:\wt c_+.$$
In a similar way we obtain  
$$
\lim _{x\to \8}x^{\alpha_2} \P (\wt W_1 <-x)= w_+c_{2,-}+w_-c_{2,+}=:\wt c_-.$$
Notice that
$$
\wt c_+ + \wt c_-=(c_{2,+}+c_{2,-})w>0.$$
If $\P (A_{22}<0)>0$ then
$$
c_{2,+}=c_{2,-}>0$$
and so
$$
\lim _{s\to \8}c_{s,-}=\lim _{s\to \8}c_{s,+}=\lim _{s\to \8}c_{2,+}w_s=c_{2,+}w >0. $$
\hfill $\qed$ 
\bigskip 

\noindent 
{\bf [\,Case $\a _1>\a_2$, $\mathcal{B}(\a _2)$\,]}
 For $\wt W_{1,0}$ we work on the expression 
\begin{align}
  \wt W_{1,0} 
 &= \sum_{i=1}^\infty \Pi_{0,2-i}^{(1)} A_{12,1-i} W_{2,-i} \nonumber \\
 &= \sum_{\ell=1}^\infty \sum_{i=1}^\ell \Pi_{0,2-i}^{(1)} 
 A_{12,1-i} \Pi_{-i,1-\ell}^{(2)} B_{2,-\ell} \nonumber \\
 &= \sum_{\ell=0}^\infty M_{\ell+1} B_{2,-\ell-1},            \label{wtW_1:seriesform}  
\end{align}
where $M_\ell$ is given in \eqref{Mn}. Here 
Fubini theorem is applicable in \eqref{wtW_1:seriesform} because for $0<\bar a<\a _2$, $\E |A_{ii}|^{\bar{\a}}<1$,
$\E |A_{12}|^{\bar{\a}}<\8$, $\E |B_2|^{\bar{\a}}<\infty$  and so $\E |\wt W_{1,0}|^{\bar \a}<\8$
due to sub-additivity for $\bar \a \le 1$ or Minkowski
inequality for $\bar \a > 1$.  

To study the tail of $\wt W_1$ we use 
	\eqref{wtW_1:seriesform}.  
Borrowing the framework of Sec. 2.2 of
 \cite{hult:samorodnitsky:2008}, where we write
\[
\mathcal F_j=\sigma
 ((\bfA_0,\bfB_0),\ldots,(\bfA_{-j},\bfB_{-j}))\ \text{and}\ \Big[ Z_j=B_{2,-j-1},\,
 \wt A_j=M_{j+1}\ \text{for}\, X=\sum_{j=0} \wt A_j Z_j \Big], 
\]
 where $\wt A_j$ is $A_j$ of \cite[Eq.\,(1.1)]{hult:samorodnitsky:2008}. It is not difficult to observe that $\wt A_j\in \mathcal F_j,\,Z_j \in \mathcal F_{j+1}$ and $\mathcal F_j$ is independent of $\sigma (Z_j,Z_{j+1},\ldots)$ for
 $j\ge 0$. Let $q=\min (\E |A_{11}|^{\a _2}, \E |A_{22}|^{\a _2})<1$. Then 
$\E|M_{\ell}|^{\a_2}=\E|\wt A_{\ell-1}|^{\alpha_2}\leq C\ell q^{\ell-1}$ and so 
$\E|\sum_{\ell=0}^\infty
 M_{\ell+1}|^{\a _2}=\E|\sum_{\ell=0}^\infty \wt A_{\ell}|^{\a _2} <\infty$. Therefore, non-zero mean condition (3.11) in
 \cite{hult:samorodnitsky:2008} is satisfied and applying Theorem 3.1
 of \cite{hult:samorodnitsky:2008} together with Remark 3.2, we obtain 
\begin{align*}
 \P(\pm \wt W_{1,0}>x) \sim \sum_{\ell=1}^\infty \big\{
\E (M_{\ell}^{\pm})^{\a _2} p_{\a _2}+\E(M_{\ell}^{\mp})^{\a _2} q_{\a _2}
\big\} x^{-\a _2}\ell_2(x). 
\end{align*}
\hfill $\qed$

\section{Proof of Theorems \ref{main1} and \ref{main}}\label{proofs}
Throughout this section, unless specified, $C,C',C_1,C_2, C_3$ denote positive constants whose values are not of interest. 
Since $\P(|\wh W_1|>x)\sim c x^{-\alpha}$, 
it suffices to prove \eqref{plusminus0} and \eqref{both} for $\wt W_1$. We further decompose $\wt W_1$ into partial
sums and study each of them. For that, given $x$, we define $n_0=n_0(x), n_1=n_1(x), n_2=n_2(x)$ and $L=L(x)$ as follows.
\begin{itembox}[l]{Indices}
\begin{equation}\label{def:notationspf}
n_0= \lfloor \rho_1^{-1}\log x \rfloor,\quad n_1=n_0 -L,\quad n_2=n_0+L,\quad L= \lfloor \ov D \sqrt{(\log \log x)\log x} \rfloor,
\end{equation}
where $\ov D$ is a sufficiently large constant.  
\end{itembox}
Indeed,
\begin{equation}\label{ovD}
\ov D^2>8C_0\max (4(\a +1)\rho _1^{-3}, \rho _1^{-1})
\end{equation}
will follow, where $C_0$ depends only on the laws of $A_{11}, A_{22}$ and it is fixed once for all in \eqref{exp}.
We write $\wt W_1$ as  
\begin{equation}\label{split}
\wt W_{1,0} = \Big( \underbrace{\sum_{i=1}^{n_1}}_{\wt Z_{n_1}}
+\underbrace{ \sum_{i=n_1+1}^{n_2}}_{\wt Z^{n_1,n_2}} + \underbrace{ \sum_{i=n_2+1}^\infty}_{\wt Z^{n_2}} \Big)\, \Pi_{0,2-i}^{(1)}
A_{12,1-i}W_{2,-i} =: \wt Z_{n_1} +  \wt
Z^{n_1,n_2} + \wt Z^{n_2}, 
\end{equation}
where 
$\wt Z_{n_1}$ is the main part and it will be proved that
\begin{equation}\label{mainbeta}
\P (\pm \wt Z_{n_1}>x)\sim x^{-\a}(\log x)^{\beta}.
\end{equation}
Here  
$\beta $ is determined by the behavior of 
$$\E (M_{n_1}^{\pm })^{\a}\sim n_1^{\beta }\sim (\log x)^{\beta},$$
which is proved separately for 
$\beta =\a , \a \slash 2$ in Theorem \ref{main1} and $\beta =1$ in Theorem \ref{main}.
For both theorems we start the proof by describing asymptotics of 
\begin{equation}\label{Mnn}
\E (M_n^{\pm})^{\a }\sim n^{\beta},\quad \mbox{as}\ n\to \8, 
\end{equation}
which is done in Lemmas \ref{lem4}, \ref{lem5} and \ref{lem2}.
The other terms in \eqref{split} are negligible.  
They are carefully handled in Section \ref{sec:neg}. 
In Lemma \ref{inftypart}, we prove that
\begin{equation*}
\P (|\wt Z^{n_2}|>x)=o(x^{-\a})\quad \mbox{as}\ x\to \8.
\end{equation*}
Once \eqref{Mnn} is obtained, Lemma \ref{middlepart} implies that
\begin{equation*}
\P (|\wt Z^{n_1, n_2}|>x)=o(x^{-\a}(\log )^{\beta })=o(\P (|\wt Z _{n_1}|>x))\quad \mbox{as}\ x\to \8.
\end{equation*}
First we complete the proof of  
Theorem \ref{main1} where the analysis of $\wt Z_{n_1}$ is considerably simpler than that for 
Theorem \ref{main}
(Section \ref{pf:main1}).  
In Section \ref{subsec:pf:main} we do the same for Theorem \ref{main}. 
Then to analyze the tail of the main part $\wt Z_{n_1}$ we exploit several auxiliary results in Section \ref{mainpart2}. 

\subsection{Proof of Theorem \ref{main1}-the main part}
\label{pf:main1}
We further divide the main part $\wt Z_{n_1}$ in \eqref{split} into two parts,  
by using the previous decomposition \eqref{def:decomp:w2} of $W_2$, 
\begin{align}
\wt Z_{n_1} 
&=\underbrace{\sum _{i=1}^{n_1}\Pi ^{(1)}_{0, 2-i}A_{12,1-i} \Pi ^{(2)}_{-i, 1-n_1}W_{2,-n_1}}_{\wt Z_{n_1,1}}+
\underbrace{\sum_{i=1}^{n_1} \Pi^{(1)}_{0,2-i} A_{12,1-i}\sum_{k=0}^{n_1-i-1} \Pi^{(2)}_{-i,1-i-k}
	B_{2,-i-k}}_{\wt Z_{n_1,2}}  \\
 &=M_{n_1}W_{2,-n_{1}}+\wt Z_{n_1,2} 
\label{def:zn12} 
\end{align}
and first our attention is focused on $\wt Z_{n_1,1}$. 
We have the following asymptotics (recall \eqref{notations1:muai1} and \eqref{notations2:muai1} for notation). 
\begin{theorem}\label{main1term} 
	Assume that $A_{11}=A_{22}\neq 0$ a.s. and that $[{\rm A1}], [{\rm A2}]$ and $\sigma^2<\infty$ hold.
	If $\mu=0$ then 
	\begin{align}
	\label{clt+}
	 \lim _{x\to \infty}\P (\pm \wt Z_{n_1,1}>x)x^{\a}(\log x)^{-\a \slash 2}&=c_2 \mathcal{C}\slash 2,  \\
	 \text{so that} \quad \lim _{x\to \infty}\P (|\wt Z_{n_1,1}|>x)x^{\a}(\log x)^{-\a \slash 2}&=c_2
	\mathcal{C}. \nonumber
	\end{align}
	If $\mu\neq 0$ then 
	\begin{align}
	\label{s+}
	 \lim _{x\to \infty}\P (\pm \wt Z_{n_1,1}>x)x^{\a}(\log x)^{-\a } &=
	\left \{
	\begin{array}{cl}
	c_{2,\pm}\mu^{\a }\rho_1^{-\a}  & \mbox{if}\quad \mu>0 \\
	c_{2,\mp}|\mu|^{\a }\rho_1^{-\a} & \mbox{if}\quad \mu <0, 
	\end{array}
	\right.
	\\
	 \text{so that} \quad \lim _{x\to \infty}\P (| \wt Z_{n_1,1}|>x)x^{\a}(\log x)^{-\a } &=c_2\rho_1^{-\a}|\mu|^{\a}.  \nonumber
	\end{align}
\end{theorem}
Moreover, with our choice of $\ov D$ by Corollary \ref{smaller}, we have 
\begin{equation*}
\P (|\wt Z_{n_1,2}|>x)=o(x^{-\a}),
\end{equation*}
and so, Theorem \ref{main1} follows. Once $Z_{n_1}$ is decomposed as in \eqref{def:zn12}, in view of \eqref{Mn} we need to handle the terms 
\begin{equation*}
\Pi ^{(1)}_{0, 2-i}A_{12,1-i} \Pi ^{(2)}_{-i, 1-n_1}
\end{equation*}
and then the estimates for them we obtain in Lemma \ref{Products} are essential. 

The remaining part of the section is devoted to the proof of Theorem \ref{main1term} and Lemmas \ref{lem4},  \ref{lem5}. 
Heuristically we will observe that 
\begin{equation*}
\P (\pm Z_{n_1,1}>x)\sim \E (M_{n_1}^{\pm})^{\a}\P (W_{2,-n_1}>x)+\E (M_{n_1}^{\mp})^{\a}\P (-W_{2,-n_1}>x).
\end{equation*}
Apparently, if $\P (A_{11}<0)>0$ only the behavior of $\E |M_{n_1}|^{\a}$ is needed which explains the content of the next lemma. 
\begin{lemma}\label{lem4}
	Assume that $A_{11}=A_{22}\neq 0$ a.s. and $[{\rm A1 \mathchar`- A3}]$. Moreover, $\mu=0$ and $\sigma^2<\infty$. 
	Then
	\begin{equation}\label{zero}
	\lim _{x\to \infty}(\log x)^{-\a \slash 2}\E |M_{n_1}|^{\a }=\mathcal{C}. 
	\end{equation}
	If additionally $\P (A_{11}>0)=1$, then 
	\begin{equation}\label{zero+}
	\lim _{x\to \infty}(\log x)^{-\a \slash 2}\E (M_{n_1}^{\pm})^{\a }=\mathcal{C} \slash 2
	.\end{equation} 
\end{lemma}
\begin{proof}
Denoting $U_{1-i}= A_{12,1-i}A_{11,1-i}^{-1}$, we observe 
\begin{align*}
\E	|M_{n_1}|^{\a }= \E \big|\Pi ^{(1)}_{0,1-n_1}\big |^{\a} |U_0+U_{-1}+\cdots +
U_{1-n_1} |^{\a}. 	\end{align*}
We study the partial sum 
$$
S_n= 
(\sigma^2 n)^{-1/2}
(U_0+U_{-1}+\cdots+U_{1-n})
$$
under change of the measure as in Lemma \ref{lemma:sreconstant}.
Notice that $(U_{-j}) _{j=0}^{\infty }$ is an iid sequence under $\P_\a$ with 
$\E_a U=\mu=0$, $\E_\a U^2=\sigma^2 <\infty$ and $\E_\a|U|^\a=\E|A_{12}|^\a<\infty$.
Indeed for $B_i \in \mathcal{B}(\mathbb{R}),\,i=1,\ldots,n$ 
\begin{align*}
\P_\a \big(\cap_{i=0}^n \{U_{-i}\in B_i\} \big) = \E|\Pi_{0,-n}^{(1)}|^\a \prod_{i=0}^n \Ind {\{U_{-i}\in B_i \}} 
= \prod_{i=0}^n \E|A_{11,-i}|^\a \Ind {\{U_{-i}\in B_i \}} = \prod_{i=0}^n \P_\a (U_{-i}\in B_i).
\end{align*}
Thus, in view of $n_1 \sim \rho_1^{-1}\log x $ it is enough to prove that
\begin{align} \label{limit:notmal}
\lim _{n\to \infty} \E_\a |S_n|^\a = \E |N|^\alpha,\qquad N\sim N(0,1).  
\end{align}
However, due to CLT \cite[Theorem 4.2]{gut:2009},  $\lim_{n\to\infty}\E _{\a}|S_n|^{\a}=\E_\alpha |N_\alpha|^\alpha$ 
where $N_\alpha$ is the standard normal w.r.t $\P_\alpha$. This is \eqref{limit:notmal}. 
For \eqref{zero+} we observe 
	\[
	\E(M_{n_1}^\pm)^\a = \E (\Pi_{0,1-n_1}^{(1)})^\a ((U_0+\cdots+U_{1-n_1})^\pm)^\a
	\]
	and by continuous mapping theorem $\lim_{n\to\infty}\E_\a (S_n^\pm)^\a= \E(N_\a^\pm)^\a$, which implies \eqref{zero+}. 
\end{proof}
\begin{lemma}\label{lem5}
	Assume that $A_{11}=A_{22}\neq 0$ a.s. and $[{\rm A1 \mathchar`- A3}]$. If $\mu\neq 0$ and 
	$\sigma^2<\infty$, then  
	\begin{equation}\label{nonzero}
	\lim _{x\to \infty}(\log x)^{-\a }\E	|M_{n_1}|^{\a }= \rho _1^{-\a}|\mu|^{\a}\quad \text{and} \quad 
	\lim _{x\to \infty}(\log x)^{-\a }\E (M_{n_1}^{\pm})^\a= \rho _1^{-\a}(\mu^{\pm})^\a.
	\end{equation}
\end{lemma}
\begin{proof}
	We follow the idea in the proof of Lemma \ref{lem4} and  
	let $S_n = (U_0+\cdots+U_{1-n}) / n $. 
	Then the first part of \eqref{nonzero} is equivalent to 
\begin{align}
\label{eq:conve:moment}
 \lim _{n\to \8}\E _{\a}|S_n |^{\a}= |\mu|^{\a}.
\end{align} 
By SLLN \cite[Theorem 4.1]{gut:2009} $S_n\to\mu\ a.s.\, \P_\a$ and $L^\a$ as $n\to\infty$. 
Then applying \cite[Theorem 4.5.4]{chung:2000}, we have \eqref{eq:conve:moment}. 
The second part of \eqref{nonzero} 
follows from the continuous mapping theorem, and we omit the details.   
\end{proof}
\begin{proof}[Proof of Theorem \ref{main1term}] 
Recall from that \eqref{def:zn12} that $\wt Z_{n_1,1}=M_{n_1}W_{2,-n_1}$, and $M_{n_1}$ and $W_{2,-n_1}$ are independent. 
For convenience we drop $-n_1$ from $W_{2,-n_1}$ and just write $W_2$. 
	Let  
	\begin{align*}
	I_{M,+} =&\P \left (W_2>x(M_{n_1}^+)^{-1}\right )x^{\a}\left (M_{n_1}^{+}\right )^{-\a},\\
	I_{M,-} =&\P \left (-W_2>x(M_{n_1}^-)^{-1}\right )x^{\a}\left (M_{n_1}^-\right )^{-\a}.
	\end{align*}
	Then
	\begin{equation*}
	\P (\wt Z_{n_1,1}>x)x^{\a}(\log x)^{-\beta}=\E I_{M,+} (M_{n_1}^+)^{\a}(\log x)^{-\beta}+\E I_{M,-} (M_{n_1}^-)^{\a}(\log x)^{-\beta} =: I_++I_-,
	\end{equation*}
	where $\beta =\a \ \mbox{or}\ \a \slash 2$. 
	
	If $\P (A_{11}<0)>0$, the Goldie constant is $c_2/2$ and thus for $\eps >0 $ and sufficiently large $T>0$,   
	$$
	|\P (\pm W_2>x)-c_2 \slash 2|<\eps \quad \mbox{for} \ x>T. 
	$$
	We claim that  
	\begin{align}
	\label{limabs0}
	\lim _{x\to \8}\big (I_++I_--c_2/2 \cdot \E |M_{n_1}|^{\a }(\log x)^{-\beta }\big )=0,
	\end{align}
	which gives the conclusion by Lemma \ref{lem4}. 
	In view of \eqref{nonzero} we have
	\begin{align*}
	&\left |I_++I_--c_2/2\cdot \E |M_{n_1}|^{\a }(\log x)^{-\beta}\right |\\
	&\leq (\log x)^{-\beta}\E \left |I_{M,+}-c_2/2\right |
	(M_{n_1}^+)^{\a }  \big( \Ind {\{M_{n_1}^+< x
		T^{-1}\}}+\Ind {\{ M_{n_1}^+ > x T^{-1}\}} \big)\\
	&\quad +(\log x)^{-\beta} \E \left |I_{M,-}-c_2/2 \right |
	(M_{n_1}^-)^{\a } \big(\Ind {\{M_{n_1}^-<xT^{-1}\}}+\Ind
	{\{M_{n_1}^->xT^{-1}\}} \big)\\
	&\leq \eps (\log x)^{-\beta}\E |M_{n_1}|^{\a }\Ind {\{|M_{n_1}|<xT^{-1}\}}+ C (\log x)^{-\beta }\E |M_{n_1}|^{\a }\Ind {\{|M_{n_1}|>xT^{-1}\}},
	\end{align*} 
	where $(\log x)^{-\beta} \E|M_{n_1}|^\a$ is bounded by Lemmas \ref{lem4} and \ref{lem5}. 
	By \eqref{tailM} below 
	\begin{equation}\label{tail0}
	\lim _{x\to \8}(\log x)^{-\beta }\E |M_{n_1}|^{\a }\Ind {\{|M_{n_1}|>xT^{-1}\}}=0.
	\end{equation}
	Then letting $x\to \8$ first and then $\eps \to 0$, we obtain \eqref{limabs0}. Now \eqref{zero} 
	yields the 
	`$+$' part of \eqref{clt+}, while  the `$+$' part of 
	\eqref{s+} with $c_{2,\pm}=c_2/2$ follows form the first part of \eqref{nonzero}. 
	The `$-$' parts 
	hold by changing signs before $W_2$ in both $I_{M,\pm}$. 
	Notice that if $\P(A_{11}<0)>0$, then we do not need to distinguish '$+$' and '$-$' parts in both \eqref{clt+} and \eqref{s+}. 

	If $\P (A_{11}>0)=1$, the Goldie's constants are $c_{2,\pm}$ and we claim that  
	\begin{equation}\label{limplus0}
	\lim _{x \to \8}\left (I_\pm-c_{2,\pm}\E (M_{n_1}^\pm)^{\a }(\log
	x)^{-\beta }\right )=0.
	\end{equation}
	Indeed, similarly as before we write 
	\begin{align*}
	&\left |I_\pm-c_{2,\pm}\E (M^\pm_{n_1})^{\a }(\log x)^{-\beta}\right |\\
	&\leq (\log x)^{-\beta}\E \left |I_{M,\pm}-c_{2,\pm}\right |
	(M_{n_1}^\pm)^{\a }  \big( \Ind {\{M_{n_1}^\pm < x
		T^{-1}\}}+\Ind {\{ M_{n_1}^\pm >x T^{-1}\}} \big)\\
	&\leq \eps (\log x)^{-\beta}\E |M_{n_1}|^{\a }+ C (\log x)^{-\beta}\E |M_{n_1}|^{\a }\Ind {\{|M_{n_1}|>xT^{-1}\}},
	\end{align*} 
	where the last term tends to $0$ as $x \to \8$ and $\varepsilon \to 0$ under \eqref{tail0}. 
	Notice that \eqref{limplus0} implies 
	\begin{equation*}
	\left |I_++I_--c_{2,+}(\log x)^{-\beta}\E (M_{n_1}^{+})^\a-c_{2,-}(\log x)^{-\beta}\E (M_{n_1}^{-})^\a \right |\to 0. 
	\end{equation*}
	Thus by \eqref{zero+} of Lemma \ref{lem4} for $\beta =\a \slash 2$ we obtain the `$+$' part of \eqref{clt+} with $c_2=c_{2,+}+c_{2,-}$.
	The `$-$' part of \eqref{clt+} follows by changing signs before $W_2$ of $I_{M,\pm}$, so that $c_{2,\pm}$ changed to $c_{2,\mp}$ in 
	\eqref{limplus0}, though these operations yield the same result. The `$\pm$' parts of \eqref{s+} are similar, but we rely on 
	the second part of \eqref{nonzero} in Lemma \ref{lem5}. 

	Now we are going to prove \eqref{tail0}. We apply Lemma \ref{Products} to the case $\rho_1=\rho_2$ and \newline
	$I_{i,k}:=|\Pi_{0,2-1}^{(1)}A_{12,1-i} \Pi_{-i,1-n_1}^{(2)}|$, i.e. $k=n_1-i$. 
	For $m\in \N$ in view of \eqref{Product1} we have
	\begin{align*}
	\P (|M_{n_1}|>xe^mT^{-1})&\leq \sum _{i=1}^{n_1}\P \left (|I_{i,k}| >xe^mT^{-1}n_1^{-1}\right )\\
	&\leq C_1n_1^{\a +2} (\log x)^{-\xi } x^{-\a} e^{-m\alpha-m\varepsilon_x }T^{\alpha +1},
	\end{align*}
	where $e^mT^{-1}n_1^{-1}$ plays the role of $T$ in \eqref{xi}. Hence 
	\begin{align*}
	\E |M_{n_1}|^{\a}\Ind {\{|M_{n_1}|\geq x T^{-1}\}}
	&\leq \sum _{m\geq 0 } \E |M_{n_1}|^{\a}\Ind {\{xe^mT^{-1}\leq |M_{n_1}|\leq xe^{m+1}T^{-1}\}}\\
	&\leq \sum _{m\geq 0}e^{\a (m+1)}x^{\a}T^{-\a}\P (|M_{n_1}|>xe^mT^{-1}) \\
	& \le C_1 T n_1^{\a+2} (\log x)^{-\xi}\sum_{m\ge 0} e^{-\varepsilon_x m} \\
	&\leq C_2 T n_1^{\a+2} (\log x)^{-\xi} \varepsilon_x^{-1}. 
	\end{align*}
	Since it follows from \eqref{xi} that
	\[
	\varepsilon_x^{-1} \le C_3 (\log x)^{1/2},
	\]
	and so by \eqref{ovD} we obtain 
	\begin{equation}\label{tailM}
	\E |M_{n_1}|^{\a }\Ind {\{|M_{n_1}|>xT^{-1}\}} \le C(\log x)^{-\xi+\alpha+5/2} T\to 0. 
	\end{equation}	
\end{proof}

\subsection{Proof of Theorem \ref{main}}
\label{subsec:pf:main}
The aim of this section is to prove the 
following theorem 
\begin{theorem}\label{mainterm} 
Under assumptions of Theorem \ref{main}
	 \begin{equation}\label{plusminus}
	 	 \lim _{x \to \infty}\P (\pm \wt Z_{n_1}>x)x^{\a}(\log
 x)^{-1}=\mathcal{D} \a \rho _1^{-1}.\end{equation}
\end{theorem}
Once \eqref{plusminus} is proved,  
Theorem \ref{main} follows. 
In this case splitting $\wt Z_{n_1}$ into $\wt Z_{n_1,1}$ and $\wt Z_{n_1,2}$, as before, is not sufficient. It turns out that both parts may contribute to the asymptotics.    
Therefore, we need to decompose $\wt Z_{n_1}$ in a different way. 
We do it in two steps to arrive finally at blocks of the type 
\begin{equation}\label{blocklenght}
Z_{h,j}=\sum _{i=h+1}^j\Pi ^{(1)}_{0,2-i}A_{12,1-i}W_{2,-i}.\end{equation}
To construct properly $Z_{h,j}$ we define  
\begin{equation}\label{K}
 K=K(x)=\lfloor  \rho _1 \rho _2^{-1} (L/2-1)\rfloor
\end{equation}
and choose
\begin{equation*}
J=J(x)\in \N \ \mbox{such that}\ JK\leq n_1<(J+1)K.\end{equation*}
 Further, let 
 \begin{equation*}
 K'=K-\lfloor K^{\theta} \rfloor 
\end{equation*}
for a fixed $0<\theta <1$. Then $j-i$ in \eqref{blocklenght} is equal to $K$ or $K^{\theta}$, which is much smaller than $n_1$. 
\vspace{2mm}  

\noindent
{\bf First decomposition.} Firstly we decompose $\wt Z_{n_1}$ as 
\begin{align*}
\wt Z_{n_1}=& \Big ( \sum _{s=0}^{J-1}
\Big(\sum _{i=sK+1}^{sK+K'} 
+ \sum _{i=sK+K'+1}^{(s+1)K} \Big)
+\sum _{i=JK+1}^{n_1} \Big ) \Pi ^{(1)}_{0,2-i}A_{12,1-i}W_{2,-i}\\
=&\sum _{s=0}^{J-1}\Big(\underbrace{Z_{sK,sK+K'}}_{R_s}+\underbrace{Z_{sK+K', (s+1)K}}_{Q_s}\Big)+
\underbrace{Z_{JK,n_1}}_{R_J}=:\sum_{s=0}^{J-1}(R_s+Q_s) +R_J,
\end{align*}
where $R_s$ are blocks of length $K'=K-\lfloor K^{\theta}\rfloor$ and $Q_s$ are those of $\lfloor K^\theta \rfloor$. Introducing shorter blocks $Q_s$, we may regard $R_s$ as nearly ``independent''. Moreover, they constitute the main part and due to ``independence'',
\begin{equation}\label{heur1}
\P \left (\pm \wt Z_{n_1}>x\right )\sim \P \left( \sum _{s=0}^{J-1}\pm R_s >x\right ) \sim \sum _{s=0}^{J-1}\P \left(\pm R_s >x\right ), 
\end{equation}
All of these approximations are heuristic and not completely exact at this stage. 
This kind of approach has already been taken in \cite{buraczewski:damek:mikosch:zienkiewicz:2013}.\\
{\bf Second decomposition.} 
Secondly, as in \eqref{decomp:wtXunders}, we apply the iteration \eqref{uniSRE2} of $W_2$ to blocks $Z_{h,j}$
and we write
\begin{align*}
Z_{h,j}=&\Pi ^{(1)}_{0,1-h}\sum _{i=h+1}^j\Pi ^{(1)}_{-h,2-i}A_{12,1-i}W_{2,-i},\qquad j\ge i\\
:=&\Pi ^{(1)}_{0,1-h}M_{h,j}W_{2,-j}+\Pi ^{(1)}_{0,1-h}Z_{h,j,2},
\end{align*}
where
\begin{align}
\label{def:M}
M_{h,j}=&\sum _{i=h+1}^j\Pi ^{(1)}_{-h,2-i}A_{12,1-i}\Pi ^{(2)}_{-i,1-j}, \\
Z_{h,j,2}=&\sum _{i=h+1}^j\Pi ^{(1)}_{-h,2-i}A_{12,1-i}\sum _{k=0}^{j-i-1}\Pi ^{(2)}_{-i,1-i-k}B_{2,-i-k}. \nonumber
\end{align}
Accordingly, $R$ and $Q$ are further decomposed as  
\begin{equation*}
R_s = R_{s,1}+R_{s,2}\quad \text{and}\quad Q_s = Q_{s,1}+Q_{s,2}, 
\end{equation*}
where for $s \le J-1$ 
\begin{align}
\label{eq:def:Rs1}
R_{s,1}=\Pi ^{(1)}_{0,1-sK}M_{sK,sK+K'}W_{2,-sK-K'}\quad &\mbox{and}\quad R_{s,2}=\Pi ^{(1)}_{0,1-sK}Z_{sK,sK+K',2}, \\
Q_{s,1}=\Pi ^{(1)}_{0,1-sK-K'}M_{sK+K',(s+1)K}W_{2,-(s+1)K}\quad &\mbox{and}\quad Q_{s,2}=\Pi ^{(1)}_{0,1-sK-K'}Z_{sK+K',(s+1)K,2},  
\nonumber
\end{align}
and for $s=J$
\begin{align}
 R_{J,1} = \Pi_{0,1-JK}^{(1)} M_{JK,n_1}W_{2,-n_1}\quad \mbox{and}\quad R_{J,2}=\Pi_{0,1-JK}^{(1)} Z_{JK,n_1,2}.
\end{align}

Notice that $\Pi ^{(1)}_{0,1-sK}, M_{sK,sK+K'}, W_{2,-sK-K'}$ are independent and $M_{sK,sK+K'}\stackrel{d}{=}M_{0,K'}=:M_{K'}$. 
Finally
\begin{equation}\label{Zn1}
\wt Z_{n_1}=\sum_{s=0}^{J-1}R_{s,1}+\sum_{s=0}^{J-1}Q_{s,1}+\sum_{s=0}^{J-1}(R_{s,2}+Q_{s,2})+R_{J}.
\end{equation}
Now we are able to make \eqref{heur1} more precise. The tail asymptotics of $\wt Z_{n_1}$ is determined by $\P \left (\sum _{s=0}^{J-1}R_{s,1}>x\right )$. Moreover,
due to separation between $R_{s,1}$ and $R_{r,1},\,s\neq r$, they are kind of    
independent (see Lemma \ref{two}), i.e.
\begin{equation}\label{sumR}
\P \left (\sum _{s=0}^{J-1}R_{s,1}>x\right )= \sum _{s=0}^{J-1} \P \left (R_{s,1}>x\right )+\ \mbox{lower order terms}.
\end{equation}
The decomposition 
\eqref{sumR} constitutes the main part of the proof of Theorem \ref{main}. 
Moreover, by Lemma \ref{block} a single block behaves as 
\begin{equation*}
\P \left (R_{s,1}>x\right )\sim x^{-\a} K,
\end{equation*}
and so
\begin{equation*}
\sum _{s=0}^{J-1} \P \left (R_{s,1}>x\right )\sim x^{-\a}(\log x).
\end{equation*}
Similarly,
\begin{equation*}
\P (R_J>x)=O(x^{-\a}K).
\end{equation*}
The other terms in \eqref{Zn1} are negligible ad they are taken care of in section \ref{mainpart2}. In Lemma \ref{rq2} and Corollary \ref{Qsum} we prove that
\begin{align*}
\P \Big (\sum_{s=0}^{J}|R_{s,2}|+\sum_{s=0}^{J-1}|Q_{s,2}| >\del  x \Big)&=o(x^{-\a}\min(\del , 1)^{-\a -1}\\
\P \Big (\sum_{s=0}^{J-1}|Q_{s,1}| >\del x \Big )&=o(x^{-\a}\log x)\min(\del , 1)^{-\a}.
\end{align*}
Therefore, 
\begin{equation*}
\P \left (\wt Z_{n_1}>x\right )= \sum _{s=0}^{J-1} \P \left (R_{s,1}>x\right )+\ \mbox{lower order terms}.
\end{equation*}
The tail of $R_{s,1}$ is described in Lemma \ref{block}. 
In view of \eqref{eq:def:Rs1} and Breiman's lemma, we have 
\begin{equation*}
\P (\pm R_{s,1}>x)\sim \E (M_{K'}^\pm)^{\a }\P (W_2>x)+\E (M_{K'}^\mp)^{\a }\P (-W_2>x),\quad \mbox{as}\ x\to \8 
\end{equation*}
(cf. \eqref{tailw1:breiman}). 

To handle $\E (M_n^{\pm})^{\a }$, we 
change the measure as in \eqref{changemeasure}. Namely, we 
consider SRE \eqref{newsre} under $\P_{\a}=\P _{\a _2}$. Again if $\P (A_{22}<0)>0$, only $\E |M_n|^{\a}$ is needed. 

\begin{lemma}\label{lem2}
Let $V=A_{11} A_{22}^{-1}$. 
	Under assumptions of Theorem \ref{main}, 
$ 0<\rho_V =\E _{\a}|V|^{\a}\log |V| <\8$ holds, 
the stationary solution $X_0$ to \eqref{newsre}  as well as the limits
\begin{equation}\label{limit1}
\lim _{x\to \8}\P (\pm X_0>x)x^{-\a}\quad \text{and} \quad \lim _{x \to \infty} (n\a)^{-1} \E |M_n|^{\a }
\end{equation}
exist. Moreover,
\begin{equation*}
c_R:=\lim _{n \to \infty} (n\a)^{-1}\E |M_n|^{\a }=\lim _{x\to \8} \rho_V \P (|X_0|>x)x^{-\a}.
\end{equation*} 
If additionally $\P (A_{22}>0)=1$ then 
\begin{equation}\label{limit2}
c_{R,\pm}:=\lim _{n \to \infty}(n\a )^{-1}\E (M_n^{\pm})^{\a }= 
\lim _{x\to \8} \rho_V \P (\pm X_0>x)x^{-\a}
\end{equation}
exists and $c_R=c_{R,+}+c_{R,-}$ is not zero iff for every $x\in \R$, $\P (A_{11}x+A_{12}=A_{22}x)<1$. 
Notice that if also $\P(A_{11}>0)=1$ in \eqref{limit2} then $c_{R,\pm}=c_R/2$. 
\end{lemma} 


\begin{proof}	
As in the proof of Lemma \ref{lemma:sreconstant}, we write 
\begin{align*}
 \E|M_{n}|^{\a }=&  \E | \Pi ^{(2)}_{0,1-n} |^{\a} 
\big|
\sum_{i=1}^{n} V_0\cdots V_{2-i} U_{1-i}
\big|^{\a} :=\E _{\a}\big|  
\sum_{i=1}^{n} V_0\cdots V_{2-i} U_{1-i}
\big|^{\a} .
\end{align*}
Now $\sum_{i=1}^{n} V_0\cdots V_{2-i} U_{1-i}$ plays the role of $\mathcal{X }_n$ in Theorem \ref{constant} and so
for \eqref{limit1} it suffices to check assumptions $\mathcal A(\alpha)$ and \eqref{perpetuity} for the recursion $X_t=V_tX_{t-1}+U_t$. 

By [A1], $\E _{\a}|V |^{\a }=\E |A _{11}|^{\a }=1$ and $\E _{\a}|V |^{\a +\eta }+ \E _{\a}|U |^{\a +\eta }<\8$ for some $\eta>0$ 
(see [A5]). A similar argument as with SRE \eqref{newsre}, $0<\rho<\infty$ follows (Take $f(\beta)=\E_\a|V|^\beta$ with $f(0)=f(\a)=1$ and 
apply convexity of $f(\beta)$ together with $f(\alpha+\eta)<\infty$). Due to [A4], $\log|V|$ is non-arithmetic. 
Moreover, $Vx+U=x \Leftrightarrow A_{11}x +A_{12}=A_{22}x$. Thus the assumptions are satisfied. 

For \eqref{limit2} we write
\begin{align*}
\E	(M_{n}^\pm)^{\a }
=& \E (\Pi ^{(2)}_{0,1-n} )^{\a}\big\{\big (
\sum _{i=1}^{n}V_0 \cdots V_{2-i} U_{1-i}\big )^{\pm} \big\}^\a	=\E_\a \big\{\big (
\sum _{i=1}^{n}V_0 \cdots V_{2-i} U_{1-i}\big )^{\pm} \big\}^\a
\end{align*}
and we proceed as before using Theorem \ref{constant}. 
\end{proof}
\begin{proof}[Proof of Theorem \ref{mainterm}] 
We are going to show that 
	\begin{equation}\label{Zasym}
	\lim _{x\to \8}\P (\pm\wt Z_{n_1}>x)x^{\a}(\log x)^{-1}=\lim _{x\to \8}
\sum _{s=0}^{J-1}\P (\pm R_{s,1}>x) x^{\a }(\log x)^{-1}=\a \rho_1^{-1}\mathcal{D}.
	\end{equation}
Since the proofs are quite similar, we only treat the positive case. 
Firstly notice that by the previous lemma 
\[
 \P(R_{J,1}>x) \le C\E|M_{JK,n_1}|^\a x^{-\a}=C\E|M_{n_1-JK}|^\a x^{-\a} \le CK x^{-\a} =o(x^{-\a}\log x).
\]
To justify  \eqref{Zasym} we are going to use auxiliary results gathered in section \ref{mainpart2}. 
In view of Lemma \ref{rq2} and Corollary \ref{Qsum} we have
	\begin{align*}
	\P (\wt Z_{n_1}>x) &\leq \P \Big (\sum _{s=0}^{J-1} R_{s,1}>(1-\eps)x\Big )
         +\P \Big (\sum _{s=0}^{J-1} Q_{s}+\sum _{s=0}^{J} R_{s,2} > \eps x \slash 2 \Big )+ \P (R_{J,1}>\eps x\slash 2)\\
	&= \P \Big (\sum _{s=0}^{J-1} R_{s,1}>(1-\eps)x\Big )+\eps ^{-\a -1}o(x^{-\a}\log x),
	\end{align*}	
where we notice that $J \le n_1 /K=o(\log x)$, and similarly 
	\begin{align*}
     \P (\wt Z_{n_1}>x) \ge 
	\P \Big (\sum _{s=0}^{J-1} R_{s,1}>(1+\eps)x\Big ) 
	-\P \Big (-\sum _{s=0}^{J-1} Q_{s} -\sum _{s=0}^{J} R_{s,2}> \eps x \slash 2\Big) 
	-\P (-R_{J,1}>\eps x\slash 2).
	\end{align*}
	Hence with some $r_{\varepsilon,x}:=\eps ^{-\a -1}o(x^{-\a}\log x)>0$, 
	\begin{multline}
	 \P \Big (\sum _{s=0}^{J-1} R_{s,1}>(1+\eps)x\Big )-r_{\varepsilon,x} \leq \P (\wt Z_{n_1}>x)
	 \leq \P \Big (\sum _{s=0}^{J-1} R_{s,1}>(1-\eps)x\Big)+r_{\varepsilon,x}
	\end{multline}
	and it is enough to prove that 
\begin{align}
\label{lsup}
	\limsup_{x\to\infty} \P \Big (\sum _{s=0}^{J-1} R_{s,1}>(1-\eps)x\Big)x^{\a}(\log x)^{-1} & \leq \a \rho _1^{-1}\mathcal{D} (1-2\eps)^{-\a }, \\
\label{linf}
        \liminf_{x\to\infty} \P \Big (\sum _{s=0}^{J-1} R_{s,1}>(1+\eps)x\Big )x^{\a}(\log x)^{-1} &\geq \a \rho _1^{-1}\mathcal{D} (1+2\eps)^{-\a }. 
\end{align}
Indeed, letting $\eps \to 0$ in \eqref{lsup} and \eqref{linf} we obtain \eqref{Zasym}. 

Choose $0<8\del =\eps <1\slash 3$ and decompose the event $\{\sum_{s=0}^{J-1} R_{s,1}>(1\pm \varepsilon) x\}$ into three ones: 
either all $|R_{s,1}|$ are smaller than $\del x$ or at least two of them are larger than $\del x$ 
or just one is larger than $\del x$. The last event is dominant. 
By Lemmas \ref{two} and \ref{allsmall},
	\begin{align}\label{reminder11}
	\P \Big ( \sum _{s=0}^{J-1}R_{s,1}>(1\pm \eps)x,\quad  \forall s \  |R_{s,1}|\leq \del x\Big) &=o(x^{-\a })\eps ^{-\a},\\
	\label{reminder22}
	\P \Big ( \sum _{s=0}^{J-1} R_{s,1}>(1\pm \eps)x,\ \ \exists r\neq u\ \  |R_{u,1}|>\del x, \ |R_{r,1}|> \del x\Big)& =o(x^{-\a })\eps ^{-\a}.
	\end{align}
Thus, suppose now that there is only one $s$ such that $|R_{s,1}|>\del x$. 
Then either $\sum _{r\neq s} |R_{r,1}|$ is larger than 
$\eps x$ or not. 
The first case is irrelevant  because
	by Lemma \ref{allsmall}, we have 
\begin{equation*}
	\P \Big (  \sum _{r=0}^{J-1} R_{r,1}>(1\pm \eps )x, |R_{s,1}|>\del x, \sum _{r\neq s} |R_{r,1}|> 
\eps x,\forall r\neq s\, |R_{r,1}|\leq \del x \Big)=o(x^{-\a })\eps ^{-\a}.
\end{equation*}
	In the second case $R_{s,1}>0$ and  we are left with disjoint sets $\wt \Omega _{s},\,s=0,...,J-1$:  
	\begin{equation*}
	\wt \Omega _{s,\pm }=\Big \{ \sum _{r=0}^{J-1} R_{r,1}>(1\pm \eps)x, R_{s,1}>\del x, \sum _{r\neq s} |R_{r,1}|\leq \eps x, \forall r\neq s\, |R_{r,1}|\leq \del x\Big \}. 
	\end{equation*}
We further define disjoint sets 
\begin{equation*}
\Omega _{s,\pm }=\Big \{ R_{s,1}>(1\pm 2\eps) x, \sum _{r\neq s} |R_{r,1}|\leq \eps x ,\forall r\neq s\, |R_{r,1}|\leq \del x \Big \},  
\end{equation*}
and then 
\begin{equation*}
\Omega _{s,+}\subset \wt \Omega _{s,\pm }\subset \Omega _{s,-}. 
\end{equation*}
We are going to prove that
	\begin{equation}\label{Omega}
	\lim _{x\to \8} \P (\Omega_{s,\pm })x^{\a}K^{-1}
=\a \mathcal{D}(1\pm 2\eps )^{-\a}
	\end{equation}
	holds uniformly in $s$.
Let us see first that \eqref{Omega} implies \eqref{lsup} and \eqref{linf} and then prove \eqref{Omega}. 
For every $\eta >1$ there is $x_0$ such that 
	\begin{equation*}
	 \P (\Omega_{s,- })x^{\a}K^{-1}\leq \eta \a \mathcal{D}(1- 2\eps )^{-\a}
	\end{equation*}
for $x\geq x_0$ and all $s$. Moreover, increasing possibly $x_0$, we may assume that for $x\geq x_0$, $(\log x)^{-1}\leq \eta \rho ^{-1}_1(JK)^{-1}$ (see \eqref{K}).
So by disjointness of the sets $\wt \Omega_s$, we have
	\begin{align*}
&	\limsup_{x\to\infty} \P \Big (\sum _{s=0}^{J-1} R_{s,1}>(1-\eps)x\Big )x^{\a}(\log x)^{-1} \\
&\leq \limsup_{x\to\infty} \P \Big (\bigcup_{s=0}^{J-1} \wt \Omega_{s,- } \Big ) x^{\a}(\log x)^{-1}\\
        &\leq \limsup_{x\to\infty} \Big(\sum_{s=0}^{J-1}\P (\Omega_{s,- } )x^{\a}K^{-1} \Big)\eta 
	\rho ^{-1}_1J^{-1}\\
	&\leq \eta ^2\a \mathcal{D} (1-2\eps)^{-\a }\rho _1^{-1}.
	\end{align*}
	Now letting $\eta \downarrow 1$ we obtain \eqref{lsup}. We proceed similarly with \eqref{linf}.
	
	Finally, we prove \eqref{Omega}.
	Notice that $\Omega_{s,\pm }\subset \{ R_{s,1 }>(1\pm 2\eps)x\}$ and
	\begin{align*}
	\{ R_{s,1 }>(1\pm 2\eps)x\} \setminus \Omega_{s,\pm }&\subset \bigcup _{r\neq s} \{ R_{s,1 }>(1\pm 2\eps)x, |R_{r,1}|>\del x\} \\
	&\quad \cup \big \{ \sum _{r\neq s} |R_{r,1}|> \eps x ,\ \forall r\neq s\, |R_{r,1}|\leq \del x \big \}.
	\end{align*}
	 Hence in view of Lemmas \ref{two} and \ref{allsmall}  
	\begin{equation}\label{difference}
	\P (\{ R_{s,1 }>(1\pm 2\eps)x\} \setminus \Omega_{s,\pm } )\leq J \delta ^{-\a}o(x^{-\a})
	\end{equation} 
	independently of $s$.
	On the other hand,
	in view of Lemma \ref{block} and definition of $\mathcal{D}$,  
	\begin{equation}
\label{limit:Rs1}
	\lim _{x\to \8}\P \left ( R_{s,1}> (1\pm 2\eps)x\right)x^{\a}K^{-1}=(1\pm 2\eps)^{-\a}\a \mathcal{D}
	\end{equation} 
	uniformly in $s$ 
	and \eqref{Omega} follows.
\end{proof}	
The next lemma gives the precise tail asymptotics of $R_{s,1}, Q_{s,1}$, which respectively yields \eqref{limit:Rs1}
in the proof of Theorem \ref{mainterm} and \eqref{asympt:Qs1} in the proof of Corollary \ref{Qsum}. 
Recall that $K=\lfloor \rho_1\rho_2^{-1}(L/2-1) \rfloor$ and 
$K'=K- \lfloor K^\theta \rfloor$ with $0<\theta<1$.
\begin{lemma}\label{block} 
Under assumptions of Theorem \ref{main} we have
	\begin{equation}\label{blockasym}
	\lim _{x\to \8} \P (|R_{s,1}|>x)x^{\a}K^{-1}=\lim _{x\to \8} \P (|Q_{s,1}|>x)x^{\a}K^{-\theta}=
	c_2c_R\a.
	\end{equation}
If additionally $\P (A_{22}<0)>0$ or $[\,\P (A_{22}>0)=1, \P (A_{11}<0)>0\,]$, then
	\begin{equation}\label{blockasym1}
\lim _{x\to \8} \P (\pm R_{s,1}>x)x^{\a}K^{-1}=c_2c_R \a/2.   
\end{equation}
If additionally $[\,\P (A_{22}>0)=1, \P (A_{11}>0)=1\,]$, then
\begin{equation}\label{blockasym2}
\lim _{x\to \8} \P (\pm R_{s,1}>x)x^{\a}K^{-1}=(c_{2,\pm }c_{R,+}+c_{2,\mp }c_{R,-})\a . 
\end{equation}
These convergences are uniform in $s$. Here as before we set $c_2=c_{2,+}+c_{2,-}$ and $c_R=c_{R,+}+c_{R,-}$.
\end{lemma}
\begin{proof}
Since $R_{s,1}$ and $Q_{s,1}$ are the same in the structure and differ only in the number of terms, we consider $R_{s,1}$ only and 
omit the proof for $Q_{s,1}$. 
First we notice that \eqref{blockasym} is implied by \eqref{blockasym1} and \eqref{blockasym2} and we prove the latter two. 
Since the expression \eqref{eq:def:Rs1} of $R_{s,1}$ is lengthy for convenience we write 
\[
 \Pi ^{(1)} _{0,1-sK} = \Pi_{1-sK},\,M_{sK,sK+K'} = M_{K'},\, W_{2,-sK-K'} =W_2\ \text{so that}\ R_{s,1}=\Pi_{1-sK} M_K' W_2. 
\]
This makes sense since $M_{sK,sK+K'} \stackrel{d}{=} M_{0,K'}:=M_{K'}$,  $W_{2,-sK-K'} \stackrel{d}{=} W_2$ (by stationarity), $\Pi _{1-sK},\ M_{K'}$ and $W_2$ are 
 mutually independent 
and $\E|\Pi_{1-sK}|^\a=1$. 
We are going to use regular variation of $W_2$ and define 
\begin{align*}
P_+=&\P \big (W_2>x((\Pi _{1-sK}M_{K'})^+)^{-1}\big )x^{\a} ((\Pi _{1-sK}M_{K'})^{+} )^{-\a},\\
P_-=&\P \big (-W_2>x((\Pi _{1-sK}M_{K'})^-)^{-1}\big )x^{\a} ((\Pi _{1-sK}M_{K'})^{-})^{-\a} 
\end{align*}
with convention that $P_{\pm}=0$ on the set $\{ \Pi_{1-sK} M_{K'}=0\}$. 
Hence
\begin{align}
\label{approx:Rs1}
\P (R_{s,1}>x)x^{\a}K^{-1}&=\E P_+ ((\Pi _{1-sK}M_{K'})^+)^\a K^{-1}+\E P_- ((\Pi _{1-sK}M_{K'})^-)^\a K^{-1} \\
&=: \ov I_++ \ov I_-. \nonumber
\end{align}
Observe that for every $\eps >0$ there is $T>0$ such that 
\begin{equation}\label{regularP}
|P_{\pm}-c_{2,\pm}|<\eps\quad \text{if}\quad x((\Pi _{1-sK}M_{K'})^{\pm}))^{-1}>T, 
\end{equation}
where $c_{2,\pm}$ are as in \eqref{tail:w_2}.
Hence one may expect that 
$\P (R_{s,1}>x)x^{\a}K^{-1}$ is approximated by
 $c_{2,\pm}\E ( (\Pi _{1-sK}M_{K'})^{\pm} )^\a K^{-1}$, as $x\to \8 $. 
We will make this intuition precise.  
\vspace{2mm}\\
{\bf Step 1.} 
We utilize the following inequalities, which depend on signs of $A_{11}$ and $A_{22}$. 
Suppose that $\P (A_{22}<0)>0$. Then $c_{2,\pm}=c_2\slash 2$ and  
 \begin{align}
&  |\P (R_{s,1}>x)x^{\a}K^{-1}-c_2(2K)^{-1}\E |M_{K'}|^{\a } | \nonumber \\
& = |\P (R_{s,1}>x)x^{\a}K^{-1}-c_2(2K)^{-1}\E |\Pi_{1-sK} M_{K'}|^{\a } | 
\label{block:ineq:1}   \\
&\leq | \ov I_+ -c_2(2K)^{-1}\E ((\Pi_{1-sK} M_{K'})^{+})^\a  |+
  | \ov I_- -c_2(2K)^{-1}\E ((\Pi_{1-sK} M_{K'})^{-})^\a  |. \nonumber 
\end{align}
If $\P (A_{22}>0)=1$ then we write
    \begin{align}
\label{block:ineq:2}    
   &|\P (R_{s,1}>x)x^{\a}K^{-1}-c_{2,+}K^{-1}\E (M_{K'}^{+})^\a  - c_{2,-}K^{-1}\E (M_{K'}^{-})^\a |\\
    &\leq |\ov I_+-c_{2,+}K^{-1}\E (M_{K'}^{+})^\a  |+
    |\ov I_- -c_{2,-}K^{-1}\E (M_{K'}^{-})^\a  |. \nonumber
\end{align}
If additionally $\P (A_{11}>0)=1$ then in \eqref{block:ineq:2} we have 
\begin{align}
\label{block:ineq:3}
  |\ov I_\pm-c_{2,\pm}K^{-1}\E (M_{K'}^\pm)^\a  |
  &=|\ov I_\pm -c_{2,\pm}K^{-1}\E ((\Pi_{1-sK} M_{K'})^\pm)^\a |.  
\end{align}
If $\P (A_{11}<0)>0$ then we write
\begin{align}\label{block:ineq:4}
|\ov I_\pm - c_{2,\pm}K^{-1}\E (M_{K'}^\pm)^\a|  
&\leq | \ov I_\pm -c_{2,\pm}K^{-1}\E ((\Pi_{1-sK} M_{K'})^\pm)^\a |\\
&+ c_{2,\pm}K^{-1}|\E ((\Pi_{1-sK} M_{K'})^\pm)^\a-\E (M_{K'}^\pm)^\a|.\nonumber
\end{align}
In {\bf Step 3} we will prove that 
\begin{equation}\label{block:ineq:6}
\lim _{x\to \8} K^{-1}|\E ((\Pi_{1-sK} M_{K'})^\pm)^\a-\E (M_{K'}^\pm)^\a|=0,
\end{equation}
provided $\P (A_{22}>0)=1, \P (A_{11}<0)>0$.
Then \eqref{block:ineq:1}, \eqref{block:ineq:3} and \eqref{block:ineq:4} may be treated in the same way because $c_{2,\pm}=c_2/2$ in \eqref{block:ineq:1}. What we need is
\begin{equation}\label{block:ineq:5}
\lim _{x\to \8} | \ov I_\pm -c_{2,\pm}K^{-1}\E ((\Pi_{1-sK} M_{K'})^\pm)^\a |=0
\end{equation}
and it will be proved in {\bf Step 4}. 
\vspace{2mm}\\
{\bf Step 2.}
Observe that $K\sim K'$ and in view of Lemma \ref{lem2}
\begin{equation}\label{equality0} 
\lim _{x\to \8}K^{-1}\E |M_{K'}|^{\a }=c_R\a \end{equation}
and if additionally $\P (A_{22}>0)=1$ then 
\begin{equation}
\label{equality1}
\lim _{x\to \8}  K^{-1} \E (M_{K'}^{\pm})^{\a }=c_{R,\pm}\a,  
\end{equation}
where $c_{R,\pm}=c_R/2$ holds when $\P (A_{22}<0)>0$. 
Now \eqref{blockasym1} for $\P (A_{22}<0)>0$ follows from \eqref{block:ineq:1}, \eqref{block:ineq:5} and \eqref{equality0}. 
If $[\,\P (A_{22}>0)=1,\,\P (A_{11}>0)=1\,]$ then, similarly, by \eqref{block:ineq:2}, \eqref{block:ineq:3}, 
\eqref{block:ineq:5} and \eqref{equality1}, 
\begin{equation*}
\lim _{x\to \8}\P (R_{s,1}>x)x^{\a}K^{-1}= (c_{2,+}c_{R,+} + c_{2,-} c_{R,-})\a,
\end{equation*}
so that \eqref{blockasym2} follows. Finally 
If $[\,\P (A_{22}>0)=1,\,\P (A_{11}<0)>0\,]$ then 
by \eqref{block:ineq:2}, \eqref{block:ineq:4}-\eqref{block:ineq:5} and \eqref{equality1} with $c_{R,\pm}=c_{R}/2$, 
the right hand side becomes
\begin{equation*}
(c_{2,+}c_{R,+} + c_{2,-} c_{R,-})\a=c_2c_R\a \slash 2, 
\end{equation*}
which is \eqref{blockasym1} under the second condition.
The proof for $-R_{s,1}$ is similar and so it is omitted. 
It suffices to change the signs of $(\Pi_{1-sK}M_{K'})$ in both $P_+$ and $P_-$ and proceed as before. 
\vspace{2mm}\\
{\bf Step 3.} 
First we show \eqref{block:ineq:6}. 
Observe that 
      \begin{align*}
      &K^{-1}\E ( M_{K'}^+)^\a =K^{-1}\E |\Pi_{1-sK}|^{\a} (M_{K'}^+)^\a \\
      &=K^{-1}\E (\Pi_{1-sK}^+)^\a (M_{K'}^+)^{\a }+K^{-1}\E (\Pi_{1-sK}^-)^\a (M_{K'}^-)^{\a }\\
      &\quad +K^{-1}\E (\Pi_{1-sK}^-)^\a (M_{K'}^+)^{\a }-K^{-1}\E (\Pi_{1-sK}^-)^\a (M_{K'}^-)^{\a }\\
      &=K^{-1}\E ((\Pi_{1-sK}M_{K'})^+)^\a \\
      &\quad +K^{-1}\E (\Pi_{1-sK}^-)^\a (M_{K'}^+)^{\a }-K^{-1}\E (\Pi_{1-sK}^-)^\a (M_{K'}^-)^{\a }.
      \end{align*}
      Since 
      \begin{equation*}
      \lim _{K\to \8}K^{-1}\E  (M_{K'}^{\pm})^{\a }=c_R\slash 2, 
      \end{equation*}
       the last row tends to $0$. Indeed, 
   \begin{align*}
   \lim _{x\to \8} K^{-1} \big| \E (\Pi_{1-sK}^-)^\a \big|  |\E (M_{K'}^+)^{\a }-\E (M_{K'}^-)^{\a } |
   \leq \lim _{x\to \8} |K^{-1}\E (M_{K'}^+)^{\a }-K^{-1}\E (M_{K'}^-)^{\a } |=0.
   \end{align*}
The same holds with $K^{-1}\E (M_{K'}^-)^\alpha$ and its approximation $K^{-1} \E((\Pi_{1-sK}M_{K'})^-)^\a$.  
\vspace{2mm}\\
{\bf Step 4.} Due to the regular variation of $W_2$
$$\sup _{u>0}P(\pm W_2>u)u^{\a}<\8.$$
Hence $P_+, P_-$ are bounded independently of $(\Pi_{1-sK} M_{K'})^\pm$. Therefore, for \eqref{block:ineq:5} we use \eqref{regularP} and so we may write 
  \begin{align*}
  & |\ov I_\pm -c_{2,\pm}K^{-1}\E ((\Pi_{1-sK} M_{K'})^\pm)^{\a } | \\
  &\leq K^{-1}\E |P_\pm- c_{2,\pm}K^{-1}| ((\Pi_{1-sK} M_{K'})^\pm)^\a \\
  &\leq \eps K^{-1} \E ((\Pi_{1-sK} M_{K'})^\pm)^\a \Ind {\{ (\Pi_{1-sK} M_{K'})^\pm < xT^{-1}\}} \\
  & \qquad +  C K^{-1} \E ((\Pi_{1-sK} M_{K'})^\pm)^\a \Ind {\{(\Pi_{1-sK} M_{K'})^\pm \geq xT^{-1}\}}\\
  &\leq \eps K^{-1} \E | M_{K'}|^{\a }+C K^{-1} \E |\Pi_{1-sK} M_{K'}|^{\a } \Ind {\{|\Pi_{1-sK}M_{K'}| \geq xT^{-1}\}}
  \end{align*}
  and $\eps $ is independent of $s$. 
  
  Hence in view of \eqref{limit1} it suffices to prove that 
  \begin{equation}\label{tailconst}
  K^{-1}\E \left | \Pi _{1-sK}M_{K'}\right |^{\a}\Ind {\{|\Pi _{1-sK}M_{K'}|\geq xT^{-1}\}}\to 0
  \end{equation}
as $x\to\infty$ uniformly in $s$. 
 We need to estimate $\P(|\Pi_{1-sK}M_{K'}|>xT^{-1} e^m)$. 
 Recall from \eqref{def:M} that $\Pi _{1-sK} M_{K'}=\Pi^{(1)}_{0,1-sK}M_{sK,sK+K'}$ is the sum of terms
 $$
 I_{i}=
\bigg \{
\begin{array}{ll}
\Pi ^{(1)}_{0,2-i}A_{12,1-i}\Pi ^{(2)}_{-i,1-sK-K'} & \text{for}\quad sK+1 \le i \le sK+K'\ \&\ s\le J-1 \\
\Pi ^{(1)}_{0,2-i}A_{12,1-i}\Pi ^{(2)}_{-i,1-n_1}  & \text{for}\quad JK+1 \le i\le n_1 \ \&\ s=J. 
\end{array}
 $$
To estimate $I_i$ we use Lemma \ref{Products}
with $I_i$ playing the role of  
$I_{i,k}$ and  
 $$
 k=
\left \{
\begin{array}{ll}
sK+K'-i \le K & \text{for}\quad sK+1 \le i \le sK+K'\ \&\ s\le J-1 \\
n_1-i \le K & \text{for}\quad JK+1 \le i\le n_1 \ \&\ s=J. 
\end{array}
\right.
 $$
Notice that since $K=\lfloor \rho_1\rho_2^{-1} (L/2-1)\rfloor$, we have 
\[
 \rho_1 i +\rho_2 k = \rho_1 n_1 + \rho_2 K \le \rho_1 n_0 -\rho_1 L +\rho_2K \le \rho_1 (n_0-L/2). 
\]
Thus, with parameters $\wt L= L/2$, $\wt{\ov D}= \ov D/2$, $\wt n_1=\lfloor n_0-\wt L \rfloor $,  
\[
 \wt \xi  =\frac{\rho _1^3 \ov D^2}{16C_0}\quad \text{and}\quad \wt \eps_x = \frac{\rho \wt L}{2C_0(n_0-\wt L)} \ge \frac{\rho_1^2}{2C_0 (\log x)^{1/2}}, 
\]
we may apply Lemma \ref{Products} (the first part of \eqref{Product1}). 
Hence 
 \begin{equation*}
 \P \left (|I_{i}|> x(TK)^{-1}e^m \right )\leq  C(\log x)^{-\wt \xi}x^{-\a}(TK)^{\a +\wt \eps_x}e^{-(\a+\wt \eps_x) m}. 
 \end{equation*}
Then 
\begin{align*}
 \P \big (| \Pi _{1-sK}M_{K'}|> xT^{-1}e^m \big )
 &\leq \sum _{i=1}^{K'} \P (| I_i|> xT^{-1}e^m {K'}^{-1}) \\
 &\leq  CK^{\a +1+\wt \eps_x }(\log x)^{-\wt \xi}x^{-\a}T^{\a + \wt \eps_x}e^{-(\a+ \wt \eps_x) m}
\end{align*}
 and 
 \begin{align*}
 \E  | \Pi _{1-sK}M_{K'} |^{\a}\Ind {\{|\Pi _{1-sK}M_{K'}|\geq xT^{-1}\}}& \le 
  \sum _{m\geq 0}\E | \Pi _{1-sK}M_{K'} |^{\a}\Ind {\{xT^{-1}e^m\leq |\Pi _{1-sK}M_{K'}|< xT^{-1}e^{m+1}\}}\\
 & \leq 
 C(\log x)^{-\wt \xi}K^{\a +2}T \sum _{m\geq 0}e^{-\wt \eps_x m}\\
 &\le C K^{\a+2} \wt \eps_x^{-1} (\log x)^{-\wt \xi}  \\
 &\leq C(\log x)^{-\wt \xi +\a \slash 2 +2}=o(1), 
 \end{align*}
in view of \eqref{ovD} and the conclusion follows. In the same way we prove the statement for $Q_{s,1}$. 
\end{proof}

\subsection{Negligible parts for Theorems  \ref{main1} and \ref{main}}
\label{sec:neg}
In this section, we study the negligible partial sums $\wt Z^{n_2},\,\wt Z ^{n_1,n_2}$ of decomposition \eqref{split} 
as well as $\wt Z_{n_1,2}$ of \eqref{def:zn12} 
(Lemma \ref{inftypart} for $\wt Z^{n_2}$, Lemma \ref{middlepart} for $\wt Z ^{n_1,n_2}$ and Corollary \ref{smaller} for $\wt Z_{n_1,2}$). 
At this point we do not distinguish between the cases $\P (A_{11}=A_{22})<1$ and $\P (A_{11}=A_{22})=1$. 
The main tool is Lemma \ref{Products} where we derive 
the behavior of the products $\Pi ^{(1)}_{0, 2-i}A_{12,1-i} \Pi ^{(2)}_{-i,1-i-k}$ 
which appear in all the negligible partial sums. This allows us to decide 
which products $\Pi ^{(1)}_{0, 2-i}A_{12,1-i} \Pi ^{(2)}_{-i,1-i-k}$ are too ``short'' or too ``long'' to play the role in the asymptotics.
\begin{lemma}
	\label{lem:lambdaii} 
Suppose $[{\rm A1}]$ and $\E|A_{ii}|^{\a+\eta}<\infty,\,i=1,2$ for some $\eta>0$. 
Given $0\leq \eps _0 <\eta$ there is $C_0>0$ such that for every $0\leq \eps \leq \eps _0$ and $i=1,2$
	\begin{equation}\label{exp} 
	\E |A_{ii}|^{\alpha \pm \eps }\leq e^{\pm \eps \rho _i+C_0\eps ^2}.
	\end{equation}
\end{lemma}
	\begin{proof} Since the proof does not depend on $i$, the index $i$ is omitted.
Let $\lambda(\beta)=\E |A|^{\beta}$. 
Notice that on $0<\beta <\alpha +\eta $ 
	\begin{align*}
	\lambda'(\beta)=\E|A|^{\beta} \log |A|\qquad \text{and}\qquad 
	\lambda''(\beta)=\E|A|^{\beta} (\log |A|)^2 
	\end{align*}
	are well defined and continuous.
	Indeed, for $\beta<\alpha+\eta$ 
        there is $M>0$ such that 
	$$|x|^\beta
	(\log |x|)^2 \le |x|^{\alpha+\eta} \vee M $$ 
        and we may apply the dominated convergence theorem. 
We have
	\begin{align*}
(\log \lambda(\beta))'=\frac{\lambda'(\beta)}{\lambda (\beta)}\qquad \text{and} \qquad 
(\log \lambda(\beta))''=\frac{\lambda''(\beta)\lambda(\beta)-\lambda'(\beta)^2}{\lambda (\beta)^2},
	\end{align*}
which are both well defined and continuous on $\beta <\alpha+\eta$ as well. Thus, 
	there exists $C_0$ such that 
	\begin{align}
	\label{secondd-lambda}
	\sup_{|\beta-\alpha|\leq \varepsilon _0 } 
(\log \lambda(\beta))'' \leq 2C_0<\infty.
	\end{align}
	Now Taylor series expansion together with $\lambda(\alpha)=1$ yields  
	$$
\log \lambda(\a \pm \eps) \leq \pm \eps \lambda'(\a)+C_0\eps ^2
$$ 
	and \eqref{exp} follows. 
\end{proof}
Recall that $n_1,n_2$ and $\ov D$ are as in \eqref{def:notationspf}, \eqref{ovD} and
$\rho _1, \rho _2$ as in \eqref{rho}. 
\begin{lemma}\label{Products}
	Suppose $[A1], [A2]$. For $i+k \leq n_1$ define 
	$$
	I_{i,k}=\big |\Pi ^{(1)}_{0, 2-i}A_{12,1-i} \Pi ^{(2)}_{-i,1-i-k}\big | \quad \text{or} \quad
	\big |\Pi ^{(1)}_{0, 2-i}A_{12,1-i} \Pi ^{(2)}_{-i,1-i-k}B_{2,-i-k}\big |.$$
	Then there is a constant $C>0$ such that for every $T>0$ and $x>1$
	\begin{equation}\label{Product1}
	\P (|I_{i,k}|>xT)\leq \left \{
\begin{array}{ll}
C(\log x)^{-\xi}x ^{-\alpha}T^{-\alpha -\eps_x } & \text{if}\ \rho _1i+\rho _2k\leq n_1\rho _1 \\
C(\log x)^{-\xi }x ^{-\alpha}T^{-\alpha +\eps_x } & \text{if}\ \rho _1i+\rho _2k\geq n_2\rho _1
\end{array}, 
\right. 
\end{equation} 
	where 
	\begin{equation}\label{xi}
	\xi = \frac{\rho ^3 _1 \ov D^2}{4C_0}\quad \mbox{and}\quad \frac{\rho _1^2}{2C_0\sqrt{\log x}}\leq \eps_x \leq 1.
	\end{equation}
\end{lemma} 
\begin{proof}
We start with the first part of \eqref{Product1}. Applying Markov inequality, we have
$$
\P (|I_{i,k}|>xT)\leq (\E |A_{11}|^{\a +\eps})^{i-1}\E |A_{12}|^{\a + \eps} (\E
|A_{22}|^{\a +\eps})^{k} (1 \vee \E |B_{2}|^{\a + \eps}) x^{-(\a +\eps) }T^{-(\a +\eps)}.$$
We further observe by Lemma \ref{lem:lambdaii} that 
\begin{align*}
\P (|I_{i,k}|>xT)
\leq & C e^{(\rho _1\eps +C_0\eps ^2)(i-1)}e^{(\rho _2\eps +C_0\eps ^2)k} x^{-(\a +\eps) }T^{-(\a +\eps)}\\
\leq & C e^{(\rho _1\eps +C_0\eps
	^2)(n_0-L)}x^{-(\a +\eps) }T^{-(\a +\eps)}\\
\le & Ce^{-\rho _1\eps L+C_0\eps ^2(n_0-L)}x^{-\a  }T^{-(\a +\eps)} , 
\end{align*}
because $e^{\rho _1\eps n_0} \le x^{\eps}$ 
by $n_0=\lfloor \rho_1^{-1}\log x \rfloor$. Finally, we minimize 
$-\rho _1\eps L+C_0\eps ^2(n_0-L)$ over $\eps$ on $\eps  \in(0, \eps_0 \wedge 1) $. 
The minimum is taken at 
\begin{align}
\label{def:epsmin}
\eps =\frac{\rho _1L}{2C_0(n_0-L)} < \eps _0 \wedge 1   
\end{align}
when $x$ is sufficiently large, and its value is
$$
-\frac{(\rho _1L)^2}{4C_0(n_0-L)}\leq -\xi \log \log x. $$
Indeed, from definitions of $n_0$ and $L$, it is not difficult to observe that
\begin{align*}
\frac{(\rho _1L)^2}{4C_0(n_0-L)} 
&\geq \frac{\rho ^3 _1 \ov D^2(\log \log x)\log x}{4C_0\log x}=\frac{\rho ^3 _1 \ov D^2}{4C_0} \log \log x, 
\end{align*}
and the first part follows.

For the second part of \eqref{Product1}, again by Markov inequality we write
$$
\P (|I_{i,k}|>xT)\leq (\E |A_{11}|^{\a -\eps})^{i-1}\E |A_{12}|^{\a - \eps} (\E
|A_{22}|^{\a -\eps})^{k} (1 \vee \E |B_{2}|^{\a - \eps}) x^{-(\a -\eps) }T^{-(\a -\eps)},$$
and we observe by Lemma \ref{lem:lambdaii} that 
\begin{align*}
\P (|I_{i,k}|>xT)
\leq & C e^{(-\rho _1\eps +C_0\eps ^2)(i-1)}e^{(-\rho _2\eps +C_0\eps ^2)k} x^{-(\a -\eps) }T^{-(\a -\eps)}\\
\leq & C e^{-n_2\rho _1\eps +C_0\eps^2(i+k)}x^{-\a +\eps }T^{-\a +\eps}\\
\leq & Ce^{-n_0\rho _1\eps}e^{-\rho _1\eps L+C_0\eps ^2(n_0-L)}x^{-\a +\eps }T^{-\a +\eps }\\
\leq & Ce^{-\rho _1\eps L+C_0\eps ^2(n_0-L)}x^{-\a  }T^{-\a +\eps }.
\end{align*}
Now minimizing the right hand side over $\eps $ as before, we reach the second part.  
\end{proof}
\begin{corollary}\label{smaller} 
 Assume $[{\rm A1}]$, $[{\rm A2}]$. Let $\wt Z_{n_1,2}$ be as in \eqref{def:zn12} and $\xi $ as in \eqref{xi}. 
	If $\rho _2\leq \rho _1$, then 
	\begin{equation*}
	\P (|\wt Z_{n_1,2}|>x)\leq Cx^{-\a}(\log x)^{-\xi +2\a +4}.
	\end{equation*}
\end{corollary}
\begin{proof} 
In view of \eqref{Product1}, using that $\rho _1i+\rho _2k\leq n_1\rho _1$, we have 
	\begin{align*}
	\P (|\wt Z_{n_1,2}|>x)\leq & \sum _{i=1}^{n_1} \sum _{k=0}^{n_1-i-1}\P
	\big (\big |\Pi ^{(1)}_{0, 2-i}A_{12,1-i} \Pi ^{(2)}_{-i,1-i-k}B_{2,-i-k}\big |> x n_1^{-2}\big )\\
	\leq & C n_1^2 (\log x)^{-\xi }x^{-\a }n^{2(\a +1)}_1 
	\end{align*}
	and the conclusion follows.
	\end{proof} 
\begin{lemma}\label{inftypart} 
	\label{lem0} Suppose ${[A1],[A2]}$. 
	There is $C>0$ such that for $\del >0$, $x>1$ 
	\begin{equation}
\label{ineq:Z_n2}
	\P (|\wt Z^{n_2}|>\del x)\leq Cx^{-\a}\min (\del ,1)^{-\a}(\log x)^{-\xi \slash 2+2\a +1},
	\end{equation}
where $\xi $ is as in \eqref{xi}.
\end{lemma} 
\begin{proof}
We may assume that $\del \leq 1$. The sum $\wt Z^{n_2}$ starts from $i>n_2$ and 
we write $i=n_0+k$, $k\geq L+1 $. Then 
	\begin{align*}
	\P  (|\wt Z^{n_2}| > \delta x )&\le  \sum_{i=n_2+1}^\infty
	\P\big (|\Pi_{0,2-i}^{(1)} A_{12,1-i}W_{2,-i}|>
	(6/\pi ^2)\cdot \delta x /(i-n_0)^2 \big)\\
& =
\sum_{k=L+1}^\infty
\P\big (|\Pi_{0,2-(n_0+k)}^{(1)} A_{12,1-(n_0+k)}W_{2,-(n_0+k)}|>
	(6/\pi ^2)\cdot \delta x /k ^2 \big) \\
& =:\sum_{k=L+1}^\infty \wt I_{n_0+k},
	\end{align*}
	By Markov inequality, for $0<\varepsilon\leq \eps _0$, we have 
	\[
 \wt I_{n_0+k} \le C \big(
	\E |A_{11}|^{\alpha-\varepsilon}
	\big)^{n_0+k-1} \E |A_{12}|^{\alpha-\varepsilon} \E
	|W_2|^{\alpha-\varepsilon}
	(\del x)^{-\alpha+\vep}k^{2(\alpha-\vep)}. 
	\]
	Due to Lemma \ref{lem:lambdaii} and inequality $e^{-\rho _1\eps n_0}\le e^{\eps \rho_1}x^{-\eps}$ we have 
	\begin{align*}
\wt I_{n_0+k} & \leq C e^{(-\rho _1\vep +C_0\vep^2)(n_0+k-1)}x^{-\a +\vep}\del ^{-\a}k^{2(\a-\vep)} \E |W_2|^{\a -\vep }\\
		  & \leq C e^{-\rho _1\vep k+C_0\vep ^2(n_0+k)}x^{-\a }\del ^{-\a}k^{2\a}\E |W_2|^{\a -\vep}.
	\end{align*}
	
	Next we evaluate terms $e^{-\rho_1 \eps k+ C_0\eps^2 (n_0+k)}$
	and $ \E|W_2|^{\alpha-\varepsilon}$. 
As done in the proof of Lemma \ref{Products}, we minimize $-\rho_1 \eps k+ C_0\eps^2 (n_0+k)$ over $\eps$. 
Increasing possibly $C_0$ we may assume that $2\eps _0C_0\slash \rho _1\geq 1$. The minimum is taken at 
\begin{align}
\label{min:at}
 	\eps =\frac{\rho _1k}{2C_0(n_0+k)}=\eps _k\leq \eps_0 
\end{align}
and the minimal value is 
$$
-\frac{(\rho _1k)^2}{4C_0(n_0+k)}\leq \left \{
\begin{array}{ll}
- (\xi/ 2) \cdot \log \log x  & \mbox{for}\ L\leq k\leq n_0 \\
- \rho _1^2k /(8C_0) & \mbox{for}\ k> n_0 
\end{array}.  
\right. 
$$
We evaluate the rate of $\E|W_2|^{\a -\eps }\to \infty$ as $\eps \to 0$. 
By regularly variation of $W_2$, 
\begin{equation*}
		\E |W_2|^{\a -\eps } \leq 1+(\a -\eps)\int _1^{\8}t^{\a-\eps-1}\P (|W_2|>t)\ dt \leq C\eps ^{-1},
\end{equation*}
where by \eqref{min:at}, $\eps^{-1}$ satisfies 
$$
	\eps ^{-1}\leq 4C_0n_0/(\rho _1 k)\quad \text{for}\ L\leq k\leq n_0
	 \quad \text{and}\quad \eps ^{-1}\leq 4C_0/ \rho _1 \quad \text{for}\ k> n_0.
$$
Collecting all the bounds above we reach 
	\begin{align}
	\sum_{k=L+1}^{n_0} \wt I_{n_0+k} & \le C (\log x)^{-\xi \slash 2}  x^{-\alpha} \del ^{-\a}\sum_{k=L+1}^{n_0} \E|W_2|^{\alpha-\varepsilon} k^{2\alpha} \nonumber \\
	&\le  C x^{-\alpha}\del ^{-\a}(\log x)^{-\xi \slash 2+2\alpha+1} 
	\label{evaluationsum1}
	\end{align}
	and 
	\begin{align}
	\sum_{k=n_0+1}^{\infty} \wt I_{n_0+k} \le C x^{-\alpha}\del ^{-\a}\sum_{k=n_0+1}^\infty k^{2\alpha} \exp\Big(
	-\frac{\rho_1^2 k}{8C_0}
	\Big)  \le C x^{-\alpha -1} (\log x)^{2\alpha}\del ^{-\a},\label{evaluationsum2}
	\end{align}
	where for the sum in the middle we apply 
	\[
	\int_{\log x}^\infty y^{2\a }e^{-y}dy = \Gamma(2\alpha+1,\log x) \le C 
	x^{-1} (\log x)^{2\alpha}  
	\]
for sufficiently large $x$. 
Since \eqref{evaluationsum2} is smaller than \eqref{evaluationsum1}, as $x\to \8$, \eqref{ineq:Z_n2} follows. 
	\end{proof}
To estimate the middle part $\wt Z ^{n_1,n_2}$, we use Petrov's large deviation theorem (Theorem \ref{petrov}) 
and so we need an additional assumption $[A4]:\,\log |A_{11}|$ is not lattice.
\begin{lemma}\label{middlepart} 
Suppose that $[A1],[A2], [A4]$ hold, $\beta \geq 0$ and 
	\begin{equation}\label{mestimate}
	\E |M_n|^{\a}=O(n^{\beta})\quad \mbox{as}\quad n\to \8.\end{equation}	
	Then 
	\begin{equation}\label{middlezet}
		P (|\wt Z ^{n_1,n_2}|>x)=o\left (x^{-\a}(\log x)^{\beta}\right ) 
		\quad \mbox{as}\quad x\to \8.
		\end{equation}
\end{lemma}
\begin{proof}
Due to stationarity we may shift indices and write 
$\wt Z ^{n_1,n_2}\stackrel{d}{=} 
\Pi ^{(1)}_{n_1,1}\wt Z_{2L}$ where $\wt Z_{2L}$ is $\wt Z_{n_1}$ of \eqref{split}
with $2L$ playing the role of $n_1$. 
Applying  \eqref{decomp:wtXunders} 
to $\wt Z_{2L}$ we obtain 
\[
 \wt Z_{2L} = \underbrace{\sum _{i=1}^{2L}\Pi ^{(1)}_{0, 2-i}A_{12,1-i}\Pi ^{(2)}_{-i, 1-2L }W_{2,-2L}}_{
\wt Z_{2L,1} =M_{2L} W_{2,-2L}
} + 
\underbrace{\sum _{i=1}^{2L}\Pi ^{(1)}_{0, 2-i}A_{12,1-i}\sum _{k=0}^{2L-i-1}\Pi ^{(2)}_{-i, 1-i-k }B_{2,-i-k}}_{\wt Z_{2L,2}}.
\]
For the second part, we notice that $\rho_1 i+\rho_2 k\le \max \rho_i \cdot 2L \le n_1\rho_1$ for $x$ large, and so applying \eqref{Product1}
 we have 
\begin{align*}
 \P (|\wt Z_{2L,2}|>x) & \le \sum_{i=1}^{2L} \sum_{k=0}^{2L-i-1} \P
\big(|\Pi^{(1)}_{0,2-i} A_{12,1-i} \Pi_{-i,1-i-k}^{(2)} B_{2,-i-k}|>x/ (2L)^2\big) \\
& \le C(2L)^2 (\log x)^{-\xi}x^{-\a}(2L)^{2(\a+1)} =o(x^{-\a}). 
\end{align*}
Concerning $\wt Z_{2L,1}$,   
by \eqref{mestimate} 
$$
\E |M_{2L}|^{\a}\leq C L^\beta =C \ov D^\beta (\log\log x \cdot \log x)^{\beta \slash 2}.
$$
Then 
\begin{equation*}
\P (|\wt Z^{n_1,n_2}|>x)
\leq \P (|\Pi ^{(1)}_{n_1,1} \wt Z_ {2L,1} |>x\slash 2)+\P (|\Pi^{(1)}_{n_1,1}\wt Z_{2L,2}|>x \slash 2)=: \wh I_1 +
\wh I_2. 
\end{equation*}
For $\wh I_1$ we apply a slightly modified version of the Breiman lemma (see e.g. Lemma 4.7 in \cite{WS}) and obtain
\begin{equation*}
\wh I_1 
\leq C\, \E |\Pi ^{(1)}_{n_1,1}|^\a\, \E |M_{2L}|^{\a}\, \P (|W_{2}|>x \slash 2)
= O(x^{-\a}(\log\log x \cdot \log x)^{\beta \slash 2}). 
\end{equation*}
The estimate for $\wh I_2$ is a little bit more complicated. We have
\begin{align*}
\wh I_2 \le &\sum _{m\geq 1}\P (e^m\leq |\Pi
^{(1)}_{n_1,1}|<e^{m+1}, \ |\wt Z_{2L,2}|>xe^{-m-1} \slash 2) +\P (|\wt Z_{2L,2}|>x e^{-1} \slash 2)\,(=o(x^{-\alpha})).
\end{align*}	
Let $m=\lfloor \rho _1 n_1 \rfloor +1+ p$. Suppose first that $p>L $ or $-\lfloor \rho _1n_1 \rfloor \leq p <-L$. 
Then using Chebychev inequality with $\a \pm\eps$ (see Lemma \ref{lem:lambdaii})  
and proceeding as in the proof of Lemma \ref{Products}, 
we obtain
\begin{equation*}
\P (|\Pi^{(1)}_{n_1,1}|\geq e^m)\leq e^{-\eps |p|+C_0\eps ^2n_1-\a m}.
\end{equation*}
Hence for such $m$ 
\begin{align*}
\P ( |\Pi^{(1)}_{n_1,1}|\geq e^m)& \P ( |\wt Z_{2L,2}|>xe^{-m-1} \slash 2)\leq C e^{-\eps |p|+C_0\eps ^2n_1-\a m}\cdot 
2^{\a} (xe^{-m})^{-\a}\\
&\le C 2^{\a}x^{-\a}e^{-\eps |p|+C_0\eps ^2n_1}
\leq C 2^{\a}x^{-\a}e^{- p^2\slash (4C_0n_1)}, 
\end{align*}
where the last inequality is obtained by minimizing over $\eps$. 

If $|p|\leq L$, we apply Theorem \ref{petrov}, 
which is due to Petrov 
\cite[Theorem 2]{petrov:1965}. Observing $m \ge n_1(\rho_1+p/n_1)$, we set 
the parameters in Theorem \ref{petrov} as 
\[
\beta =\a,\ n=n_1,\ A_i=|A_{11,i}|,\ c=\rho_1,\ \text{and}\ \gamma_n=p/n_1. 
\]
Since we have $[A1]$ and $\E \log |A_{11}|<c=\rho_1$ by convexity, the conditions of Theorem \ref{petrov} are satisfied. 
Since $\Lambda (\a )=0$, 
\begin{equation*}
\P (|\Pi^{(1)}_{n_1,1}|\geq e^m)\leq C_1 n_1^{-1/2} e^{-\a m - C_2 p^2\slash n_1 }
\end{equation*}
and thus
\begin{align*}
\P ( |\Pi^{(1)}_{n_1,1}|\geq e^m) \P ( |\wt Z_{2L,2}|>xe^{-m-1} \slash 2)&
\leq C_1 n_1^{-1/2}e^{-\a m - C_2 p^2\slash n_1 }\cdot 2^{\a}\left (xe^{-m}\right)^{-\a}\\
	&=C_1 2^{\a}x^{-\a}n_1^{-1/2}e^{- C_2 p^2\slash n_1 }.
\end{align*}	
Finally, summing up over $p$ we obtain
\begin{align*}
\wh I_2 
&\leq C2^{\a}x^{-\a}\Big (\sum _{|p|>L}e^{- p^2\slash (4C_0n_1)}
+ \sum _{|p|\leq L}n_1^{-1/2}e^{- C_2 p^2\slash n_1}\Big )\\
	&\leq C2^{\a}x^{-\a}\Big (\int_L^\infty e^{-x^2/(4C_0 n_1)} dx 
+n_1^{-1/2}\int _0^{L}e^{-C_2 x^2/n_1}\ dx\Big) \\
 & \le C x^{-\a}\Big ( 
(2C_0n_1)^{1/2} \int_{L(2C_0n_1)^{-1/2}} e^{-x^2/2}dx +C'
\Big) \\
& \le C x^{-\a}\Big ( 2C_0n_1/L \cdot e^{- L^2/(4 C_0n_1)} +C'
\Big),
\end{align*}
where in the last step, we apply the well-known inequality $\int_{x}^\infty e^{-t^2/2}dt \le x^{-1} e^{-x^2/2}$ to the integral. 
Since $2C_0n_1/L \le C(\log x)^{1/2}$ and 
\[
 e^{-L^2 /(4C_0n_1)} \le e^{-\rho_1 \ov D^2/(4C_0) \cdot \log \log x}  \le (\log x)^{-\rho_1 \ov D^2 /(4C_0)}.
\] 
If $\ov D$ satisfies \eqref{ovD} then 
$\wh I_2=O (x^{-\a})$. Thus \eqref{middlezet} follows. 
\end{proof}
\subsection{Auxiliary results for Theorem \ref{mainterm}}\label{mainpart2}
In this section we prove that 
\begin{equation*}
\P \left (\sum_{s=0}^{J-1}|R_{s,2}|+\sum_{s=0}^{J-1}|Q_{s,2}|>x\right )=o(x^{-\a}),
\end{equation*}
in Lemma \ref{rq2}, as well as we analyze the behavior of
\begin{equation}\label{RQ}
\sum_{s=0}^{J-1}R_{s,1}\quad \mbox{and}\quad  \sum_{s=0}^{J-1}Q_{s,1}.\end{equation}
It turns out that for each of the sums in \eqref{RQ}, the rule of a single jump works; probability of $|R_{s,1}|, |R_{r,1}|$, $s\neq r$ being large at the same time or probability that all $|R_{s,1}|$ are small, is of order
$o(x^{-\a})$, see Lemmas \ref{two} and \ref{allsmall}. This is due to a kind of ``independence'' obtained by separation of indices in $R_{s,1}$ and $R_{r,1}$ 
provided \eqref{ovD}.
Therefore,
$$
\P(\pm \sum_{s=0}^{J-1} R_{s,1}>x) \sim \sum_{s=0}^{J-1} \P(\pm R_{s,1}>x)\quad \text{and} \quad \P(\pm \sum_{s=0}^{J-1} Q_{s,1}>x) \sim \sum_{s=0}^{J-1} \P(\pm Q_{s,1}>x)$$
and the latter is proved in Corollary \ref{Qsum} to be of order $o(x^{-\a }\log x )$. 
\begin{lemma}\label{rq2} 
	Assume $[A1],[A2]$ and $0\le \del \leq 1$. Then
		\begin{align*}
	\P \big ( \sum _{s=0}^J |R_{s,2}|>\del x  \big)=o(x^{-\a})\del ^{-\a -1}\quad \text{and}\quad 
	\P \big ( \sum _{s=0}^{J-1}|Q_{s,2}|>\del x  \big)=o(x^{-\a})\del ^{-\a -1}
\end{align*}
as $x\to \8$.
\end{lemma}
\begin{proof} 
We start with inequality 
\begin{equation*}
\P \big ( \sum _{s=0}^J |R_{s,2}|>x  \big)\leq  \sum _{s=0}^J\P \big(| R_{s,2}|> x J^{-1} \big), 
\end{equation*}
and observe that $R_{s,2}$ is the sum of at most $K^2$ (actually $(K')^2$) terms of the type
	$$
\breve I_{i,k}=\Pi ^{(1)}_{0,2-i}A_{12,1-i}\Pi ^{(2)}_{-i,1-i-k}B_{2,-i-k}$$
	with indices 
\begin{align*}
sK+1\leq i\leq sK+K',\quad k\leq sK+K'-i-1     &\quad \text{if}\quad s\leq J-1, \\
JK+1\leq i\leq n_1,\quad k\leq n_1-i-1          &\quad \text{if}\quad s=J.  
\end{align*}
Hence we further obtain 
$$\P \big (| R_{s,2}|> \del x J^{-1}  \big )\leq \sum _{i,k}\P \big (|\breve I_{i,k}| \del x J^{-1} K^{-2}  \big ).$$   
We will apply Lemma \ref{Products} in the present setting. Since $K \leq \rho _1 \rho _2^{-1} (L/2-1)$, it follows 
that 
\begin{align*}
\rho _1i+\rho _2k\leq &\rho _1n_1+\rho_2(K-1)
=\rho _1n_0-\rho_1L+\rho_2K\leq \rho _1(n_0-L/2-1). 
\end{align*}
Take $\wt L=L/2$, $\wt{\ov D}= \ov D/2$, $\wt n_1=\lfloor n_0-\wt L \rfloor $ and then $\rho _1i+\rho _2k\leq \rho _1\wt n_1$.
It is not difficult to observe that the proof of Lemma \ref{Products} does not change with this setting. 
Now the first part of \eqref{Product1} with $\wt \xi = \frac{\rho _1^3\wt{\ov D}^2}{4C_0} =\frac{\rho _1^3 \ov D^2}{16C_0}$ yields 
\begin{equation*}
\P \big (|\breve I_{i,k}|> \del x J^{-1} K^{-2}  \big )\leq C(\log x)^{-\wt \xi}x^{-\a}(JK^2)^{\a +1}\del ^{-\a -1}.
\end{equation*}
Finally, noticing that $JK\le \rho_1^{-1} \log x$ and $K^{\a+2}\le C(\log x)^{\a/2+2}$, we obtain 
\begin{equation*}
\P \big ( \sum _{s=0}^J |R_{s,2}|>x  \big )\leq 
JK^2 \P \big (|\breve I_{i,k}|> \del x J^{-1} K^{-2} \big ) \le 
C(\log x)^{-\wt \xi +2\a +4}x^{-\a}\del ^{-\a -1}
\end{equation*}
and the conclusion follows provided $\ov D$ in $\wt \xi$ is large enough. In the same way we prove the inequality for $Q_{s,2}$. 
\end{proof}
Next we prove that probability of $R_{s,1}, R_{r,1},\,s\neq r$ being large at the same time is of smaller order. 
\begin{lemma} \label{two} 
Suppose that $[A1], [A2]$ are satisfied and 
	$\del _1, \del _2\leq 1$. 
If $s\neq r$ then  
		\begin{align*}
		\P (|R_{s,1}|>\del _1x, \ |R_{r,1}|>\del _2x)&=o(x^{-\a})\left ( \del _1^{-\a }+\del _2^{-\a }\right),\\
		\P (|Q_{s,1}|>\del _1x, \ |Q_{r,1}|>\del _2x)&=o(x^{-\a})\left ( \del _1^{-\a }+\del _2^{-\a }\right), 
			\end{align*}
			uniformly in $s$ and $r$.
\end{lemma}
	
\begin{proof} 
Let $p=\lfloor \rho _2^{-1}\log x +L\rfloor $. Then 
\[
 W_{2,-2K-K'} = \big(\underbrace{\sum_{k=0}^{p-1}}_{P_{s,1}} +\underbrace{\sum_{k=p}^\infty}_{P_{s,2}} \big)
\Pi^{(2)}_{-sK-K',1-sK-K'-k} B_{2,-2K-K'-k}=: P_{s,1} +P_{s,2},  
\]
and, in view of of \eqref{eq:def:Rs1}, we have 
\begin{equation*}
	R_{s,1}=\Pi _{0,1-sK}^{(1)}M_{sK,sK+K'}W_{2,-sK-K'}=\Pi _{0,1-sK}^{(1)}M_{sK,sK+K'}(P_{s,1}+P_{s,2}).
\end{equation*}
	First we prove that there is $C$ such that for every $x>1$, every $\del _1\leq 1$ and all $s$, 
	\begin{equation}\label{P2}
	\P (|\Pi _{0,1-sK}^{(1)}M_{sK,sK+K'}P_{s,2}|>\del _1x)\leq C\del _1^{-\a }x^{-\a}(\log x)^{-\xi +2\a +2},
	\end{equation}
	where $\xi =\frac{\rho _2^3 \ov D^2}{8C_0}$. Proceeding exactly as in the proof of Lemma \ref{inftypart}, we have
	\begin{equation}\label{tail2}
	\P (|P_{s,2}|>\del x)\leq C\del ^{-\a }x^{-\a}(\log x)^{-\xi +2\a +1}
	\end{equation}
	for $\del \leq 1$ (To follow the proof of Lemma \ref{inftypart}, 
	$\wt Z^{n_2}$ is replaced with $P_{s,2}$, and $n_2$ with $p=\lfloor \rho _2^{-1}\log x +L\rfloor $,
	 $\Pi ^{(2)}$ plays the role of $\Pi ^{(1)}$ and $B_2$ the role of $A_{12}W_2$).
	Moreover, in view of \eqref{tail2},   
	\begin{align*}
	& \P (|\Pi _{0,1-sK}^{(1)} M_{sK,sK+K'}P_{s,2}|>\del _1x) \\ 
        & \leq \P (|P_{s,2}|>\del _1x) 
	+\P (|P_{s,2}|>|\Pi _{0,1-sK}^{(1)}M_{sK,sK+K'}|^{-1}\del _1x, |\Pi _{0,1-sK}^{(1)}M_{sK,sK+K'}|\geq 1 )\\
	&\leq C\del _1 ^{-\a }x^{-\a}(\log x)^{-\xi +2\a +1} (1+\E|\Pi _{0,1-sK}^{(1)}M_{sK,sK+K'}|^{\a} )\\
	&\leq C\del _1^{-\a }x^{-\a}K'(\log x)^{-\xi +2\a +1}, 
	\end{align*}
 where Breiman's lemma is applied in the second step and \eqref{limit1} in the last. 
 Thus \eqref{P2} follows. 
Without loss of generality we may assume $s<r$ and we proceed to evaluate 
	\begin{equation*}
I'=\P (|\Pi _{0,1-sK}^{(1)}M_{sK,sK+K'}P_{s,1}|>\del _1x,\ |\Pi _{0,1-rK}^{(1)}M_{rK,rK+K'}P_{r,1}|>\del _2x),
	\end{equation*}
since our target is bounded as 
\begin{align}
\label{eq:pf:r1r2}
 \P(|R_{s,1}|>2\del _1x, \ |R_{r,1}|>2\del _2x) &\le I'+ \P(|\Pi _{0,1-sK}^{(1)}M_{sK,sK+K'}P_{s,2}|>\del _1x) \\
&\quad + \P(|\Pi _{0,1-rK}^{(1)}M_{rK,rK+K'}P_{r,2}|>\del _2x). \nonumber
\end{align}
Notice that the number of terms in $M_{sK,sK+K'} P_{s,1}$ or $M_{rK,rK+K'} P_{r,1}$ is at most $Kp$.  
For indices $sK+1\leq i_1< sK+K',\ rK+1\leq i_2 <rK+K'$ and $0\leq j_1, j_2\leq p-1$, we consider the events 
	\begin{align*}
	J_{i_1,j_1}:=&|\Pi ^{(1)}_{0,2-i_1}A_{12,1-i_1}\Pi ^{(2)}_{-i_1,1-sK-K'-j_1}B_{2,-sK-K'-j_1}|>\frac{\del _1x}
	{Kp},\\
	J_{i_2,j_2}:=&|\Pi ^{(1)}_{0,2-i_2}A_{12,1-i_2}\Pi ^{(2)}_{-i_2,1-rK-K'-j_2}B_{2,-rK-K'-j_2}|>\frac{\del _2x}
	{Kp}. 
	\end{align*}
Hence, Markov inequality yields  
	\begin{equation*}
\P \Big(
J_{i_1,j_1} > \frac{\del _1x}{Kp}, J_{i_2,j_2}> \frac{\del _2x}{Kp} 
\Big) \le 
	\P \Big (J_{i_1,j_1}J_{i_2,j_2}>\frac{\del _1\del _2 x^2}{K^2p^2}
	 \Big)\leq \E (J_{i_1,j_1}J_{i_2,j_2})^{\a \slash 2} (\del _1\del _2)^{-\a \slash 2} x^{-\a}
(Kp)^{\a}.
	\end{equation*}
To estimate the expectation in the above formula we write $(J_{i_1,j_1}J_{i_2,j_2})^{\a \slash 2}$ as the product of two i.i.d. random 
products 
and $\E (J_{i_1,j_1}J_{i_2,j_2})^{\a \slash 2}$ is written as the product of expectations of 
variables grouped in the same index. 

 For an index $m \notin \m=\{ 1-i_1, 1-i_2, -sK-K'-j_1, -rK-K'-j_2\}$ the
 terms related to $A_{11,m}$ and $A_{22,m}$ in each product are of the form
	\begin{equation*}
	|A_{11,m}|^{\a \slash 2}, |A_{11,m}|^{\a}, |A_{22,m}|^{\a \slash 2}, |A_{22,m}|^{\a} 
	\quad \mbox{or} \quad |A_{11,m}A_{22,m}|^{\a \slash 2}.
	\end{equation*}
	Moreover,
\begin{align*}
	&\E |A_{11,m}|^{\a \slash 2}<1, \quad \E |A_{22,m}|^{\a \slash 2}<1, \quad \E |A_{11,m}A_{22,m}|^{\a \slash 2}<1,
\end{align*}
where the third inequality, 
for $A_{11}\neq A_{22}$, follows from the strict inequality of Schwartz.

For an index $m \in \m $ 
expectations are finite
because
\[
\E|A_{12}|^\a (|A_{11}|^{\a/2}+|A_{22}|^{\a/2}) <\infty \quad \text{and}\quad \E|B_2|^{\a/2}
(|A_{11}|^{\a/2}+|A_{22}|^{\a/2}+|A_{12}|^{\a/2}+|B_2|^{\a/2})<\infty.
\]
Notice that in $J_{i_1j_1}$ at least $i_1-1$ of $A_{11}$ terms exist and in $J_{i_2 j_2}$ at least $i_2-1$, while the number of 
$A_{22}$ terms in both also depend on $0\le j_\ell \le p-1\,(\ell=1,2)$ and could possibly be zero. 
Thus, the number of types 
$|A_{11,m}|^{\a/2}$, $|A_{11,m} A_{22,m}|^{\a/2}$ in each product is at least 
$i_2-i_1\geq K^{\theta}+(r-s-1)K$. Since $J_{i_1,j_1}, J_{i_2,j_2}$ may be chosen in at most $(Kp)^2$ ways, 
for $\gamma=\min \{\E|A_{11,m}|^{\a/2},\E|A_{11,m}A_{22,m}|^{\a/2}\}$, 
\begin{equation}
\label{ineqlemmatwo}
I' \leq C\g  ^{K^{\theta}+(r-s-1)K }(\del _1\del _2)^{-\a \slash 2}x^{-\a} (Kp)^{2+\alpha}. 
\end{equation}
Since 
$$\gamma^{K^{\theta}} (Kp)^{2+\alpha} \le C \gamma^{K^{\theta}} K^{3(2+\alpha)}=o\left ((\log x)^{-\xi +2\a +2}\right ), \quad \mbox{as}\quad x\to \8 ,$$
recalling \eqref{eq:pf:r1r2}, 
from \eqref{P2}, \eqref{ineqlemmatwo} and \eqref{ovD} it follows that
\begin{equation}\label{Rtwo}
\P (|R_{s,1}|>\del _1 x, |R_{r,1}|>\del _2 x )\leq C \left ( \del _1 ^{-\a}+\del _2 ^{-\a} \right )(\log x)^{-\xi \slash 2}x^{-\a }.
\end{equation}	
In the same way we prove the statement for $Q_{s,1}$.  
\end{proof}
Finally, we show that probability that all blocks are very small is of smaller order which, together with the previous lemma, means that asymptotics is given by one block being large.
\begin{lemma}\label{allsmall}
Suppose that $[A1], [A2]$ are satisfied, $\del <1$ 
	 and $0<8 \del \leq \eps $. Then 
\begin{align}
\label{event:R}
\P \Big (\sum _{s=0}^{J-1} |R_{s,1}|>\eps x,\quad \forall s\ |R_{s,1}|\leq \del x\Big) &=o(x^{-\a})\del ^{-\a}\\ 
\P \Big (\sum _{s=0}^{J-1} |Q_{s,1}|>\eps x,\quad \forall s\ |Q_{s,1}|\leq \del x\Big)&=o(x^{-\a})\del ^{-\a} \nonumber 
\end{align}
as $x\to \8$.
\end{lemma}
\begin{proof}
We will prove \eqref{event:R} only for $R$-blocks. For $Q$-blocks the proof is  similar. 
Assume that $\sum _{s=0}^{J-1} |R_{s,1}|>\eps x$, and split
the event $\{s:|R_{s,1}|\le \delta x\}$ into 
\begin{equation*}
\I _j=\{ s: e^{-j}\del x<|R_{s,1}|\leq e^{-j+1}\del x\},\quad j=1,2,\ldots 
\end{equation*}
There should be $j$ such that $\I_j$ has at least 
$e^{j-1} \eps /(2j^2 \delta):=n(j)$ 
elements 
since $\# \I_j \le n(j)$ for all $j$ implies 
\begin{equation*}
\sum _{s=0}^{J-1}|R_s|=\sum _{j\geq 1}\sum _{s\in \I_j}|R_s|\leq \sum_{j\ge 1} \eps x/(2j^2) <\eps x. 
\end{equation*}
Thus the event of \eqref{event:R} is included in $\cup_j \{\# \I_j\ge n(j)\}$. 
Moreover, $$2< 4 e^{j-1}/j^2\le n(j)\le J\leq C\sqrt{\log x}$$ implies 
\begin{equation}\label{number}
 \# \I_j \ge 3\quad  \mbox{and}
\quad e^j\leq C\log x.\end{equation} 
Let $s,r\in \I _j$ such that $r-s \ge n(j)-1$ and then  
\[
 \{\# \I_j \ge n(j)\} \subset \{|R_{s,1}| > \delta e^{-j}x,\,|R_{r,1}|>\delta e^{-j} x\}. 
\]
Now applying  \eqref{Rtwo}, 
in view of \eqref{number}, we obtain for $x\ge 2$ 
\begin{align*}
\P (|R_{s,1}|>\del e^{-j}x, \ |R_{r,1}|>\del e^{-j}x)&\leq C
 \del  ^{-\a}e^{j\a }x^{-\a} (\log x)^{-\xi \slash 2}\\
 &\leq C
 \del  ^{-\a}x^{-\a} (\log x)^{-\xi \slash 2+\a}.
\end{align*}
Since we may choose $s,r$ in at most $n_1^2=O((\log x)^2)$ ways, 
\begin{equation*}
\P (\# \I _j \ge n(j))\leq C
\del  ^{-\a}x^{-\a} (\log x)^{-\xi \slash 2+\a +2}
\end{equation*}
and so by \eqref{number}, 
\begin{align*}
\P \Big (\sum _{s=0}^{J-1} |R_s|>\eps x,\ \forall s \ |R_s|\leq \del x\Big)
&\le \sum_{j\ge 1} \P(\# \I_j \ge n(j)) \\
&\le C
\del  ^{-\a}x^{-\a} (\log x)^{-\xi \slash 2+\a +3}
\end{align*}
because $j\leq \log (C\log x).$ 
\end{proof}
\begin{corollary}\label{Qsum} 
Under assumptions of Theorem \ref{main}, for $\eps >0$ we have 
	\begin{equation}\label{mainQ}
		\P (\sum _{s=0}^{J-1}|Q_{s,1}|>\eps x)=o(x^{-\a}\log x)\min(\eps ,1) ^{-\a}
\end{equation}
as $x\to \8$.
\end{corollary}

\begin{proof}
We may assume that $\eps \leq 1$. Choose $\del =\eps \slash 16$. Similarly as in
\eqref{reminder11} and \eqref{reminder22} we decompose the event $\{\sum_{s=0}^{J-1}|Q_{s,1}| >\eps x\}$
into three patterns: either all $|Q_{s,1}|$ are smaller than $\del x$ or there are at least two of them which are larger than $\del x$ 
or just one is larger than $\del x$. By Lemmas \ref{two} and \ref{allsmall} we have
\begin{equation*}
\P \big ( \sum _{s=0}^{J-1}|Q_{s,1}|>\eps x,\ \ \forall s \ \ |Q_{s,1}|< \del x \big)=o(x^{-\a })\eps ^{-\a},
\end{equation*} 
\begin{equation*}
\P \big ( \sum _{s=0}^{J-1} |Q_{s,1}|>\eps x,\ \ \exists r\neq s \ \ |Q_{s,1}|>\del x,  |Q_{r,1}|> \del x\big)=o(x^{-\a })\eps ^{-a}.
\end{equation*}
Suppose now that there is only one $s_0$ such that $|Q_{s_0,1}|>\del x$. 
Then either $\sum _{s\neq s_0} |Q_{s,1}|$ is larger than $\eps x/2$ or not. 
In the first case again by Lemma \ref{allsmall} with $\eps$ replaced by $\eps/2$,  
\begin{equation*}
\P  ( \sum _{s\neq s_0} |Q_{s,1}|> \eps x /2,\ \ \forall s\neq s_0\ |Q_{s,1}|< \del x )=o(x^{-\a })\eps ^{-a}.
\end{equation*}
In the second case 
\begin{equation}
\label{essential}
\Big \{ \sum _{s=0}^{J-1} |Q_{s,1}|>\eps x, |Q_{s_0,1}|>\del x, \sum _{s\neq s_0} |Q_{s,1}|\leq \frac{\eps x}{2}\Big \}\subset
\Big \{  |Q_{s_0,1}|>\frac{\eps x}{2}, \sum _{s\neq s_0} |Q_{s,1}|\leq \frac{\eps x}{2}\Big \} 
\end{equation} 
and for different $s_0$ the sets on the right hand side of \eqref{essential} are disjoint.
But in view of \eqref{blockasym} in Lemma \ref{block}
\begin{equation}
\label{asympt:Qs1}
\sum _{s=0}^{J-1}\P (  |Q_{s_0,1}|> \eps x/2 )\leq C(\eps x)^{-\a}K^{\theta} J \leq C (\eps x)^{-\a}
K^{\theta -1}\log x 
\end{equation}
and \eqref{mainQ} follows.	
\end{proof}

\section{On tail behavior of univariate SRE}
\label{sec:uni:SRE}
The main result of this section is an alternative formula for Goldie constants (Theorem \ref{constant}). 
	We start with a lemma that summarizes the content of 
 \cite[Theorem 5]{kesten:1973}, \cite[Theorem 4.1]{goldie:1991} and \cite[Theorem 3]{grey:1991}. For a review see also
 Theorems 2.4.3, 2.4.4 and 2.4.7 in  
\cite{buraczewski:damek:mikosch:2016}.

\begin{lemma}
	\label{lemma:4447BDM}
	Let $((A_t,B_t))_{t\in \Z}$ be an $\R^2$-valued iid sequence and consider SRE 
	\begin{align}
\label{def:uni:sre}
	X_t = A_t X_{t-1}+B_t, \quad t\in \Z. 
	\end{align}
Suppose that either $\mathcal A (\a )$ or $\mathcal B (\a )$ from Section \ref{mainresults} holds.
%
Then there is a unique stationary causal solution $X_t$ to \eqref{def:uni:sre} and $X\eqd X_t$ satisfies the 
stochastic fixed point equation 
\begin{align}
\label{def:uni:sfe}
X \eqd AX+B.
\end{align}
	Moreover, there
	exist constants $c_\pm$ such that 
	\begin{align*}
	\P(\pm X>x) \sim \bigg \{
	\begin{array}{ll}
	c_\pm x^{-\alpha} & \text{if\ $\mathcal A(\a)$ holds} \\
	c_\pm x^{-\a }\ell (x) & \text{if\ $\mathcal B(\a )$ holds}
	\end{array},
	\end{align*}
	as $x\to\infty$, where constants are given by 
\begin{align}
\label{constants:uni}
\begin{split}
\mathcal A(\a ).\quad & c_{\pm} =
\bigg \{
\begin{array}{ll}
(\alpha \rho)^{-1} \E [((AX+B)^{\pm})^\a-((AX)^{\pm})^\a] &  \mbox{if}\quad  \P (A\geq 0)=1 \\
(2\alpha \rho)^{-1}\E[|AX+B|^\alpha - |AX|^\alpha]  & \mbox{if}\quad \P (A< 0)>0
\end{array},  \\
\mathcal B(\a ). \quad & c_{\pm} = \frac{1}{2}\Big\{
\frac{1}{1-\E|A|^\a } \pm
\frac{ p_\a-q_\a }{1-\E (A^+)^\a+\E (A^-)^\a}
\Big\}
\end{split}
\end{align}
with $\rho =\E |A|^\alpha \log |A| >0$. Finally, $c_++c_->0$ in all cases. 
\end{lemma}
For the proof of Theorem \ref{mainterm} we need an alternative expression for $c_{\pm}$:
\begin{theorem}\label{constant}
	Suppose that the assumptions of Lemma \ref{lemma:4447BDM} are satisfied
	and 
	 \begin{equation}\label{perpetuity}
	 \E |A|^{\a +\eta }<\infty , \quad \E |B|^{\a +\eta }<\infty 
	 \end{equation}
	 for a strictly positive $\eta $.
	 Let $\Pi_{1,k}:=A_1\cdots A_k$ for $k\ge 1$ and $\Pi_{1,0}=1$. Moreover, 
	\begin{equation*}
	\X _n=\sum _{i=1}^n
 \Pi_{1,i-1}B_i
, \quad \X _n^+=\max (\X _n,0), \quad \X _n^-=-\min (\X _n,0).
	\end{equation*}
	Then
	\begin{equation}\label{limits}
	c_{+}=\lim _{n\to \8} (\a \rho n)^{-1}\E (\X _n^+)^{\a},\quad
	c_{-}=\lim _{n\to \8} (\a \rho n)^{-1}\E (\X _n^-)^{\a},
	\end{equation}
where $c_{\pm}$ are those of \eqref{constants:uni} for $\mathcal A(\a )$. 
	\end{theorem}
\begin{remark}
	Under assumption $A\geq 0$ a.s. \eqref{limits} was proved in \cite{buraczewski:damek:zienkiewicz:2018} and then condition \eqref{perpetuity} may be replaced by a weaker one: $\E |A|^{\a}\log |A|<\8$.
	\end{remark} 
\begin{proof}
	We prove \eqref{limits} for $c_+$, the proof for $c_-$ is similar.
Let $\delta>0,\,\beta \in(1/2,1),\,m_1=\lfloor \rho n -n^{\beta}\rfloor$ and $m_2=\lfloor \rho n +n^{\beta}\rfloor$. 

First we show that  
\begin{align}\label{conlarge}
\lim _{n\to \8} n^{-1}\E \X_n^{\a}\Ind {\{\X_n >e^{m_2}\}} &=0, \\
\lim _{n\to \8} n^{-1}\E \X_n^{\a }\Ind {\{ e^{m_1}< \X_n \leq e^{m_2} \}}&=0, \label{conmiddle} \\
\lim _{n\to \8} n^{-1}\E \X_n^{\a}\Ind {\{0< \X_n\leq e^{n^{1/2} } \}}&=0,\label{consmall}
\end{align}
which reduce \eqref{limits} to 
\begin{equation}\label{limits1}
c_{+}=\lim _{n\to \8} (\a \rho n)^{-1}\E \X _n^{\a}\Ind {\{ e^{n^{1/2} }<\X_n\leq e^{m_1} \}}.
\end{equation}
For \eqref{conlarge} we notice that for $m\ge m_2$, there is $C_1>0$ such that $n\leq m \rho^{-1}-C_1m^{\beta}$.
Then by Lemma \ref{Products} 
\begin{align*}
\P (\X_n>e^m) & \leq \sum_{k=1}^n \P\big(|\Pi_{1,k-1} B_k|>e^{m} (n^2 \cdot \pi^2/6)^{-1} \big) \\
& \le C n^{1+2(\alpha+\eps)} \E|\Pi_{1,k-1} B_k|^{\a+\eps} e^{-(\a+\eps)m} \\
& \le C m^{1+2(\a+\eps)} e^{C_0 \eps^2 (m\rho^{-1} -C_1m^\beta)-\eps \rho C_1m^{\beta}}e^{-\a m}. 
\end{align*}	
Minimizing the quantity in the exponential w.r.t. $\eps$ we have 
\[
 - \frac{C_1^2\rho^2m^\beta}{4C_0 (m\rho^{-1}-C_1m^\beta)} \le -C_2 m^{2\beta-1}\quad \text{at}\quad \underline{\eps} 
= \frac{C_1\rho m^\beta}{2C_0 (m\rho^{-1}-C_1m^\beta)}. 
\]
Hence
\[
\P (\X_n >e^m) \leq C e^{-\a m} m^{-2}, 
\]
and so 
\[
\E \X_n^{\a }\Ind {\{e^m<\X_n \leq e^{m+1}\}}\leq C e^{\a } m^{-2 }.
\]
Summing up over $m\geq m_2$ we obtain \eqref{conlarge}.	
For \eqref{conmiddle} we consider the solution $\X=\sum _{i=1}^{\8}|\Pi_{1,k-1} B_i|$ of SRE $X_t=|A_t|X_{t-1}+|B_t|$. 
In view of Lemma \ref{lemma:4447BDM}
\begin{align*}
	\E \X_n^{\a}\Ind {\{e^{m_1}< \X_n \leq e^{m_2}\}}\leq \sum _{m=m_1}^{m_2}e^{(m+1)\a }
\P (\X_n >e^m) \leq C(m_2-m_1+1)e^{\a }\leq Cn^{\beta }=o(n). 
	\end{align*}
In a similar way to \eqref{conmiddle} we obtain 
\eqref{consmall}. In this case $\E \X_n^\a \Ind {\{0< \X_n \le e^{n^{1/2}}\}}=O(n^{1/2})$. 

For \eqref{limits1} let $N_1=\lfloor n^{1/2}\delta^{-1} \rfloor$,  $N_2=\lfloor m_1 \delta^{-1} \rfloor$ and
define the sets 
\begin{equation*}
W_m=\{ e^{m\del }< \X_n \leq e^{(m+1)\del}\},\ N_1 \leq m\leq N_2-1, \quad W_{N_2}=\{ e^{N_2\del }< \X_n \leq e^{m_1}\} 
\end{equation*}
and it is enough to prove that  
\begin{equation}\label{conmain}
I(n) :=  (n\a \rho)^{-1} \sum _{m=N_1}^{N_2-1} \E \X_n^{\a}\Ind {\{W_m\}} -c_+ \to 0 \quad \text{as}\quad n\to \infty,
\end{equation}
since, as above,
\begin{equation*}
\E \X_n^{\a} \Ind {\{W_{N_2}\}}\leq Ce^{m_1\a - N_2\del \a}\leq Ce^{\del \a}. 
\end{equation*}  

In order to prove \eqref{conmain} we show that for fixed $\eps>0$ and for sufficiently large $n$, 
\begin{equation}\label{sets}
    \big |\P (W_m)e^{m \del \a}-c_+ (1-e^{-\del \a} ) \big|<\eps 
\end{equation}
uniformly in $N_1\leq m\leq N_2-1$. 
First we see that \eqref{sets} implies \eqref{conmain} and then we will prove \eqref{sets}. 
We write 
\begin{align*}
I(n) &\leq 
 (n\a \rho)^{-1} \Big |\sum _{m=N_1}^{N_2-1}\E \big (\X_n^{\a}-e^{m\del \a}\big )\Ind {\{W_m\}}\Big |
 +\Big |(n\a \rho)^{-1}\sum _{m=N_1}^{N_2-1} e^{m\del \a}\P (W_m) -c_+\Big |\\
 &= :I_1(n)+I_2(n).
    \end{align*}
    In view of \eqref{sets}, $\P (W_m)\leq C(\del +\eps)e^{-m\del \a}$ and so
    \begin{align*}
    I_1(n)&\leq C (n\a \rho)^{-1} \sum _{m=N_1}^{N_2-1} \left (e^{(m+1)\del \a}-e^{m\del \a}\right )(\del +\eps)e^{-m\del \a}\\
  &\leq C (n\a \rho)^{-1} (N_2-N_1) \del \a (\del +\eps)\leq C (\del +\eps).
    \end{align*}
Moreover, 
\begin{equation*}
    I_2(n)\leq (n\a \rho)^{-1} \sum _{m=N_1}^{N_2-1} |e^{m\del \a}\P (W_m)-c_+(1-e^{-\del \a}) |+  | (n\a \rho)^{-1}(N_2-N_1)c_+(1-e^{-\del \a})-c_+ |
    \end{equation*}
    and 
    \begin{equation*}
    \lim _{n\to \8}\frac{N_2-N_1}{n\rho }= \del ^{-1}.
    \end{equation*}
    Hence 
    \begin{equation*}
    \lim _{n\to \8}| (n\a \rho)^{-1}(N_2-N_1)c_+(1-e^{-\del \a})-c_+ |=O(\delta).
    \end{equation*}
    and by \eqref{sets}
    \begin{equation*}
    \limsup _{n\to \8}I_2(n)\leq \lim _{n\to \8}\frac{N_2-N_1}{n\a \rho}\eps = \eps (\del \a)^{-1}.
    \end{equation*}
Correcting above bounds, we have  
\[
\limsup_{n\to \8}I(n)\leq \eps (\delta \a)^{-1}+ C(\delta +\eps). 
\]
Hence letting $\eps \to 0$ and then $\del \to 0$ we obtain \eqref{conmain}. 
   
Now we prove \eqref{sets}. Let $\X=\sum _{i=1}^{\8}\Pi_{1,i-1} B_i$ and $\CY_n=\X-\X_n$. 
Proceeding as in the proof of Lemma \ref{inftypart} ($\CY_n$ plays the role of $\wt Z^{n_2}$) we can prove that there is $C_1>0$ such that 
        \begin{equation}\label{caly}
        \P \left (|\CY_n|> n^{-1} e^{m\del}\right )\leq C_1 e^{-m\del \a}n^{-1} 
        \end{equation} 
        for $N_1=O(n^{1/2})\leq m \leq N_2=O(n)$, where $\delta x$ of \eqref{ineq:Z_n2}
is replaced by $n^{-1} e^{m\del}$. 
We write
        \begin{align*}
 J(m)=\P (W_m)e^{m\del \a}-c_+\left (1-e^{-\del \a}\right )& \le \P \left (e^{m\del}<\X-\CY_n\leq e^{(m+1)\del }\right )e^{m\del \a}-c_+\left (1-e^{-\del \a}\right )\\
        	&\leq \P \left (e^{m\del} (1-n^{-1} )< \X\leq e^{(m+1)\del} (1+n^{-1})\right )e^{m\del \a}
        	\\
        	&\quad +\P \left (|\CY_n|>n^{-1}e^{m\del}\right )e^{m\del \a}-c_+\left (1-e^{-\del \a}\right ).
        	\end{align*}
        But given $\eps $, for sufficiently large $n$, we have
        \begin{align*}
        |\P \left (\X> e^{m\del}(1-n^{-1} )\right )e^{m\del \a}(1-n^{-1} )^{\a}-c_+|&<\eps, \\
        |\P \left (\X> e^{(m+1)\del}(1+n^{-1} )\right )e^{(m+1)\del \a}(1+n^{-1} )^{\a}-c_+|&<\eps .
        \end{align*}
        Hence
        \begin{equation*}
        J(m)\leq c_+ (1-n^{-1} )^{-\a} -c_+(1+n^{-1} )^{-\a}e^{-\del \a} -c_+\left (1-e^{-\del \a}\right )+\eps \left ((1-n^{-1} )^{-\a} +(1+n^{-1} )^{-\a} \right ) 
        \end{equation*}	
        and letting $n \to \8$ we obtain
        \begin{equation*}
        \limsup _{n\to \8}\left (\P (W_m)e^{m\del \a}-c_+\left (1-e^{-\del \a}\right )\right )\leq \eps.
        \end{equation*}
For the opposite inequality, notice that 
for $n$ large enough, $1+n^{-1}\leq e^{\del}(1-n^{-1})$, and so we may consider 
\[
 \{e^{m\delta} (1+n^{-1})<\X \le e^{(m+1)\delta }(1-n^{-1})\}\cap \{|\CY_n|<n^{-1}e^{m\delta}\} \subset W_m.
\]
Hence 
\begin{equation*}
\P (W_m)e^{m\del \a}\geq \P \big(e^{m\del} (1+n^{-1} )<\X \leq e^{(m+1)\del}(1-n^{-1} )\big )e^{m\del \a}-\P(|\CY_n|>n^{-1}e^{m\delta}) e^{m\delta\a}.
        \end{equation*}
Proceeding as above we have 
\begin{equation*}
\liminf_{n\to \8}\big( \P (W_m)e^{m\del \a}-c_+ (1-e^{-\del \a})\big )\geq \eps.
\end{equation*} 
\end{proof}

For a positive random variable $A$ let $\Lambda (\b )=\log \E A^{\b }$. Suppose that $\Lambda $ is well defined 
for $0\leq \b < \b _0 \leq \8 $. Then so are $\Lambda '$ and $\Lambda ''$. Let $\lambda =\sup_{\b <\b_0}\Lambda '(\b )$  
and $\sigma(\beta)=\Lambda ''(\beta) $. 
The following uniform large deviation theorem is due to \cite[Theorem 2]{petrov:1965}.
\begin{thm}[Petrov (1965)]\label{petrov} 
	Suppose  that $c$ satisfies $ {\mathbb E} \left[ \log A \right] < c < \lambda $, and suppose that $\delta(n)$ is an arbitrary function satisfying  $\lim_{n \to \infty} \delta(n) = 0$.
	Also, assume that the law of $\log A$ is non-lattice.
	Then with 
	$\beta $ chosen such that $\Lambda '(\beta )=c$, we have that
	\begin{align*} 
	{\mathbb P} &  \big( \log A_1+\dots +\log A_n > n(c + \gamma_n) \big )\nonumber\\
	& \quad\quad = \frac{1}{\beta\sigma(\beta) \sqrt{2\pi n}} \exp\Big\{-n \Big( \beta(c+\gamma_n) - \Lambda(\beta) + \frac{\gamma_n^2}{2\sigma^2(\beta)}
	\big(1 + O(|\gamma_n| ) \big) \Big) \Big\} (1+o(1))
	\end{align*}
	as $n \to \infty$,
	uniformly with respect to $c$ and $\gamma_n$ in the range
	\begin{equation} \label{petrov-0}
	{\mathbb E} \left[ \log A \right] + \eps \le c \le \lambda - \eps \quad \text{\rm and} \quad |\gamma_n| \le \delta(n),
	\end{equation}
	where $\eps >0$.
\end{thm}
\begin{remark}{\rm
		In \eqref{petrov-0}, we may have that $\sup \{ \b : \b \in dom (\Lambda) \} = \infty$ or ${\mathbb E} 
\left[ \log A \right] =-\infty$.  In these cases, the quantities
		$\infty -\varepsilon$ or  $-\infty -\varepsilon$ should be interpreted as arbitrary positive, respectively negative, constants.}
\end{remark} 
\noindent 
{\bf The list of frequently used symbols.}
\begin{enumerate}
\item $C,C',C_1,C_2, C_3$: positive constants whose values are not of interest
\item $C_0$ the constant defined in \eqref{exp}
\item $c _{R,\pm}= \lim _{n\to \8}(\a n)^{-1} \E (M_n^{\pm})^{\a }\quad \mbox{and}\quad  c_R=\lim _{n\to \8} (\a n)^{-1}\E |M_n|^{\a }>0$
\item $c_{2,\pm}$: tail constants of $W_2$ 
\item $c_{1,\pm}$: tail constants of $\wh W_1$
\item $\ov c_\pm,\,\wt c_\pm$: tail constants of $W_1$ in Table \ref{theorem:main1}
\item $\mathcal{C}= \sigma^\a \rho_1^{-\a/2} \E|N|^\a$
\item $ \mathcal D = c_2c_R/2\ \text{or}\ c_{2,\pm} c_{R,+}+c_{2,\mp}c_{R,-}$
\item $\wt{\ov D}=\ov D/2$
\item $\ov D$ the defined in \eqref{ovD} 
\item $\E _{\a _2}[Z] = \E  \big[|\Pi ^{(2)}_{0,-n} |^{\a _2} Z\big]$
\item $\mathcal F_n=\sigma((\bfA_i,\bfB_i)_{-n \le i\le 0})$
\item $J\in \N$: $JK\leq n_1<(J+1)K$
\item $K=\lfloor  \rho _1 \rho _2^{-1} (L/2-1)\rfloor$  
\item $K'=K-\lfloor K^{\theta} \rfloor$ for a fixed $0<\theta <1$ 
\item $L= \lfloor \ov D \sqrt{(\log \log x)\log x} \rfloor$
\item $\wt L=L/2$
\item 
$M_n=\sum_{i=1}^n \Pi_{0,2-i}^{(1)}A_{12,1-i}
\Pi_{-i,1-n}^{(2)}$ \eqref{Mn} 
\item $n_0= \lfloor \rho_1^{-1}\log x \rfloor$
\item $ n_1=n_0 -L,\quad n_2=n_0+L$
\item $\wt n_1=\lfloor n_0-\wt L \rfloor $ 
\item $p=\lfloor \rho _2^{-1}\log x +L\rfloor $
\item $U_i= A_{12,i}A_{22,i}^{-1}$, $U=A_{12}A_{22}^{-1}$
\item $V_i = A_{11,i}A_{22,i}^{-1}$, $V=A_{11}A_{22}^{-1}$
\item $ w_{n,\pm}=\E (M_n^{\pm})^{\a _2} \quad \text{and}\quad w_n=\E |M_n|^{\a _2}$ 
\item $w=\lim _{n\to \8} w_n >0\quad \mbox{and}\quad w _{\pm}=\lim _{n\to \8} w_{n,\pm}$
\item $\mu= \E A_{11}^{-1} A_{12} |A_{11}|^\a$  
\item $\xi = \frac{\rho ^3 _1 \ov D^2}{4C_0}$
\item $\wt \xi =\frac{\rho _1^3\wt{\ov D}^2}{4C_0} =\frac{\rho _1^3 \ov D^2}{16C_0}$
\item $  \Pi_{t,s}^{(i)} =\Pi _{j=s}^t A_{ii,j},\,t\ge s,\,i=1,2\quad
 \mathrm{and}\quad \Pi_{t,s}^{(i)}=1,\,t<s\quad \mathrm{and}\quad \Pi_t^{(i)}=\Pi_{t,1}^{(i)}$
\item $\rho_i=\E|A_{ii}|^\alpha \log |A_{ii}|,\,i=1,2$
\item $\sigma^2 =\E (A_{12}A_{11}^{-1})^2 |A_{11}|^\a$
\end{enumerate}
\bigskip

\noindent {\bf Acknowledgments}
We thank both referees for careful reading and helpful suggestions improving considerably the paper. 
E. Damek is grateful to Jacek Zienkiewcz for all the conversations on SRE they had in the past. 
Although they did not concerned this paper directly, they were an invaluable source of inspiration for her.
 E. Damek was partly supported by the NCN grant UMO-2019/33/B/ST1/00207.
M. Matsui's research is partly supported by the JSPS Grant-in-Aid for Scientific Research C
(19K11868).

{\small

\end{document}




\bibitem{basrak:segers:2009}
{\sc Basrak, B. and Segers, J.}\ (2009) Regularly varying multivariate time
series.  {\em Stochastic Process. Appl.} {\bf 119}, 1055--1080.

\bibitem{bingham:goldie:teugels:1987}
{\sc Bingham, N.H., Goldie, C.M.\ and Teugels, J.L.}\ (1987)
{\em Regular Variation.} Cambridge University Press, Cambridge (UK).


\bibitem{bollerslev:1990}
{\sc Bollerslev, T.}\ (1990)
Modelling the coherence in short-run nominal exchange rates: a
multivariate 
generalised ARCH model. {\em Review of Economics and Statistics} {\bf
	72}, 498--505.

\bibitem{boman:lindskog:2009}
{\sc Boman, J. and Lindskog, F.}\ (2009)
Support theorems for the Radon transform and Cram\'er-Wold theorems.
{\em J. Theoret. Probab.} {\bf 22}, 683--710.

\bibitem{breiman:1965}
{\sc Breiman, L.}\ (1965)
On some limit theorems similar to the arc-sin law. 
{\em Theory Probab. Appl.} {\bf 10}, 323--331.


\bibitem{davis:hsing:1995}
{\sc Davis, R.A. and Hsing, T.}\ (1995)
Point process and partial sum \con\ for weakly dependent
\rv s with infinite variance. {\em Ann. Probab.} {\bf 23}, 879--917.

\bibitem{davis:mikosch:2009}
{\sc Davis, R.A. and Mikosch, T.}\ (2009)
The extremogram: a correlogram for extreme events. 
{\em Bernoulli} {\bf 15}, 977--1009.

\bibitem{davis:mikosch:basrak:1999}
{\sc Davis, R.A., Mikosch, T. and Basrak, B.}\ (1999)
Sample ACF of multivariate stochastic recurrence equations With
application to GARCH, Preprint. 

\bibitem{davis:mikosch:cribben:2012}
{\sc Davis, R.A., Mikosch, T. and Cribben, I.}\ (2012)
Towards estimating extremal serial dependence via the boostrapped extremogram.
{\em J. Econometrics} {\bf 170}, 142--152.

\bibitem{davis:mikosch:zhao:2013}
{\sc Davis, R.A., Mikosch, T. and Zhao, Y.}\ (2013)
Measures of serial extremal dependence and their estimation. 
{\em Stochastic Process. Appl.} {\bf 123}, 2575--2602.


\bibitem{fernandez:muriel:2009}
{\sc Fern\'andez, B. and Muriel, N.}\ (2009)
Regular variation and related results for the multivariate GARCH(p,q)
model with constant conditional correlations.
{\em J. Multivariate Anal.} {\bf 100}, 1538--1550.



\bibitem{goldie:1991}
{\sc Goldie, C.M.}\ (1991)
Implicit renewal theory and tails of solutions of random equations.
{\em Ann. Appl. Probab.} {\bf 1}, 126--166.


\bibitem{jeantheau:1998}
{\sc Jeantheau, T.}\ (1998)
Strong consistency of estimators for multivariate ARCH models. 
{\em Econometric Theory} {\bf 14}, 70--86. 


\bibitem{krengel:1985}
{\sc Krengel, U.}\ (1985)
Ergodic Theorems.
With a supplement by Antoine Brunel.
Walter de Gruyter \& Co., Berlin.




\bibitem{mikosch:starica:2000}
{\sc Mikosch, T. and St\u{a}ric\u{a}, C.}\ (2000)
Limit theory for the sample autocorrelations and
extremes of a GARCH(1,1) process. 
{\em Ann. Statist.} {\bf 28}, 1427--1451.

\bibitem{mikosch:samorodnitsky:tafakori}
{\sc Mikosch, T., Samorodnitsky, G. and Tafakori, L.}\ (2013) 
Fractional moments of solutions to stochastic recurrence equations.
{\em J. Appl. Probab.} {\bf 50}, 969--982.


\bibitem{resnick:1987}
{\sc Resnick, S.I.}\ (1987)
{\em Extreme Values, Regular Variation, and Point Processes.}
Sprin\-ger, New York.
\bibitem{resnick:2007}
{\sc Resnick, S.I.}\ (2007)
{\em Heavy-Tail Phenomena: Probabilistic and Statistical Modeling.}
Springer, New York.


\end{thebibliography}}

\end{document}